\documentclass[12pt]{amsart}
\usepackage{amsmath}
\usepackage{amsthm}
\usepackage{amssymb}
\usepackage{amscd}
\usepackage{amsfonts}
\usepackage{amsbsy}
\usepackage{epsfig}
\usepackage{graphicx}
\usepackage{tikz}
\usepackage{tikz-cd}
\usetikzlibrary{arrows}
\usepackage[T1]{fontenc}

\newtheorem{theorem}{Theorem}[section]
\newtheorem{corollary}[theorem]{Corollary}

\newtheorem{lemma}[theorem]{Lemma}
\newtheorem{remark}[theorem]{Remark}
\newtheorem{question}{Question}
\newtheorem{definition}[theorem]{Definition}
\newtheorem{example}[theorem]{Example}

\theoremstyle{definition}

\def\ie{{\em i.e.,} }
\def\eg{{\em e.g.} }

\newfont\bbf{msbm10 at 12pt}
\def\eps{\varepsilon}
\def\phi{\varphi}
\def\R{{\mathbb R}}

\def\N{{\mathbb N}}
\def\Z{{\mathbb Z}}
\def\cN{{\mathcal N}}

\def\E{{\mathcal E}}

\def\D{{\mathcal D}}

\def\diam{\mbox{\rm diam}\,}

\def\theta{\vartheta}

\def\UIL{\underleftarrow\lim(I,T)}
\def\ovl{\overleftarrow}

\def\chain{{\mathcal C}}
\def\mesh{\mbox{mesh}}

\def\eps{\varepsilon}

\begin{document}

\title{Planar embeddings of chainable continua}

\author{Ana Anu\v{s}i\'c, Henk Bruin, Jernej \v{C}in\v{c}}

\address[A.\ Anu\v{s}i\'c]{Departamento de Matem\'atica Aplicada, IME-USP, Rua de Mat\~ao 1010, Cidade Universit\'aria, 05508-090 S\~ao Paulo SP, Brazil}
\email{anaanusic@ime.usp.br}
\address[H.\ Bruin]{Faculty of Mathematics,
University of Vienna,
Oskar-Morgenstern-Platz 1, A-1090 Vienna, Austria}
\email{henk.bruin@univie.ac.at}
\address[J.\ \v{C}in\v{c}]{AGH University of Science and Technology, Faculty of Applied Mathematics, al. Mickiewicza 30, 30-059 Krak\'ow, Poland. -- and -- National Supercomputing Centre IT4Innovations, Division of the University of Ostrava, Institute for Research and Applications of Fuzzy Modeling, 30. dubna 22, 70103 Ostrava, Czech Republic}
\email{jernej@agh.edu.pl}
\thanks{AA was supported in part by Croatian Science Foundation under the 
project IP-2014-09-2285.
HB and J\v{C} were supported by the FWF stand-alone project P25975-N25. J\v{C} 
was partially supported by the FWF Schr\"odinger Fellowship stand-alone project J-4276 and by 
University of Ostrava grant IRP201824 ``Complex topological
structures''.
We gratefully acknowledge the support of the bilateral grant \emph{Strange 
Attractors and Inverse Limit Spaces},  \"Osterreichische
Austauschdienst (OeAD) - Ministry of Science, Education and Sport of the 
Republic of Croatia (MZOS), project number HR 03/2014.}
\date{\today}


\begin{abstract}
We prove that for a chainable continuum $X$ and every $x\in X$ with only finitely many coordinates contained 
in a zigzag there exists a planar embedding 
$\phi:X\to \phi(X)\subset\R^2$ such that $\phi(x)$ is accessible, partially answering a question of 
Nadler and Quinn from 1972. Two embeddings $\phi,\psi:X \to \R^2$ are called strongly equivalent
if $\phi \circ \psi^{-1}: \psi(X) \to \phi(X)$ can be extended to a homeomorphism
of $\R^2$. We also prove that every nondegenerate indecomposable chainable continuum can be 
embedded in the plane in uncountably many ways that are not strongly equivalent.
\end{abstract}

\maketitle

\section{Introduction}

It is well-known that every chainable continuum can be embedded in the plane, 
see \cite{Bing}. 
In this paper we develop methods to study nonequivalent planar embeddings, 
similar to methods used by Lewis in \cite{Lew} and Smith in \cite{Sm} for the study of planar 
embeddings of the pseudo-arc. 
Following Bing's approach from \cite{Bing} (see Lemma~\ref{lem:patterns}), we 
construct nested intersections 
of discs in the plane which are small tubular neighborhoods of polygonal lines obtained 
from the bonding maps. 
Later we show that this approach produces all possible planar embeddings of chainable 
continua which can be 
covered with planar chains with \emph{connected} links, see Theorem~\ref{thm:allemb}. 
From that we can produce uncountably many  nonequivalent
planar embeddings of the same chainable continuum.

\begin{definition}\label{def:equivembed}
	Let $X$ be a chainable continuum.
	Two embeddings $\phi,\psi:X \to \R^2$ are called  {\em equivalent} if
	there is a homeomorphism $h$ of $\R^2$ such that $h(\phi(X)) = \psi(X)$.
They are {\em strongly equivalent}
	if $\psi \circ \phi^{-1}: \phi(X)\to \psi(X)$ can be extended to a 
homeomorphism of $\R^2$. 
	%(See also Remark~\ref{rem:otherdef}).
\end{definition}

That is, equivalence requires some homeomorphism between $\phi(X)$ and $\psi(X)$ 
to be extended to $\R^2$ 
whereas strong equivalence requires the homeomorphism $\psi \circ \phi^{-1}$
between $\phi(X)$ and $\psi(X)$ to be extended to $\R^2$.

Clearly, strong equivalence implies equivalence, but in general not the 
other way around,
see for instance Remark~\ref{rem:n_emb}. We say a nondegenerate continuum is {\em indecomposable}, if 
it is not the union of two proper subcontinua.

\begin{question}\label{q:uncountably}
Are there uncountably many nonequivalent planar embeddings of every chainable 
indecomposable continuum?
\end{question}

This question is listed as Problem 141 in a collection of 
continuum theory problems from 1983 by Lewis \cite{LewP} and was also posed by 
Mayer in his dissertation in 1983 \cite{MayThesis} (see also \cite{May}) using the standard definition of equivalent embeddings.

We give a positive answer to the adaptation of the above 
question using strong equivalence, see Theorem~\ref{thm:Mayer}.
If the continuum is the inverse limit space of a unimodal map and not 
hereditarily decomposable, then
the result holds for both definitions of equivalent, see 
Remark~\ref{rem:otherdef}.

In terms of equivalence, this generalizes the result in \cite{embed}, 
where we prove that every unimodal inverse limit 
space with bonding map of positive topological entropy can be embedded in the 
plane in uncountably 
many nonequivalent ways. The special construction in \cite{embed} uses 
symbolic techniques 
which enable direct computation of accessible sets and prime ends (see \cite{AC}). 
Here we utilize a more direct geometric approach.

One of the main motivations for the study of planar embeddings of tree-like continua is 
the question of whether the
\emph{plane fixed point property} holds. The problem is considered to be one of the most 
important open problems in continuum theory. 
Is it true that every continuum $X \subset \R^2$ not separating the plane has the fixed point 
property, 
\ie every continuous $f: X\to X$ has a fixed point? There are examples of 
tree-like continua 
without the fixed point property, see \eg Bellamy's example in \cite{Bell}. 
It is not known whether Bellamy's example can be embedded in the plane. 
Although chainable continua are known to have the fixed point property (see 
\cite{Ha}), insight 
in their planar embeddings may be of use to the general setting of tree-like 
continua.

Another motivation for this study is the following long-standing open problem.
For this we use the following definition.

\begin{definition}
	Let $X\subset\R^2$. We say that $x\in X$ is {\em accessible} (from the 
complement of $X$) if there exists an arc $A\subset\R^2$ such that $A\cap 
X=\{x\}$.
\end{definition}

\begin{question} [Nadler and Quinn 1972, pg. 229 in \cite{Nadler}] \label{q:NaQu} Let $X$ 
be a chainable continuum and $x\in X$. Can $X$ be embedded in the plane such 
that $x$ is accessible? 
\end{question}

We will introduce the notion of a \emph{zigzag} related to the admissible 
permutations of graphs of bonding maps 
and answer Nadler and Quinn's question in the affirmative 
for the class of \emph{non-zigzag} chainable continua (see 
Corollary~\ref{cor:nonzigzag}). 
From the other direction, a promising possible counterexample 
to Question~\ref{q:NaQu} is the one suggested by Minc (see 
Figure~\ref{fig:Minc} and the description in \cite{Minc}). 
However, the currently available techniques are insufficient to determine whether the point $p\in X_M$ can be 
made accessible or not, even with the use of thin embeddings, see 
Definition~\ref{def:thin}. 

Section~\ref{sec:notation} gives basic notation,
and we review the construction of natural chains in Section~\ref{sec:chains}.
Section~\ref{sec:permuting} describes the main technique of permuting branches 
of graphs of linear interval maps. In Section~\ref{sec:stretching} we connect 
the techniques developed in Section~\ref{sec:permuting} to chains.
Section~\ref{sec:emb} applies the techniques developed
so far to accessibility of points of chainable planar continua; this is the content of 
Theorem~\ref{thm:algorithm} which is used as a technical tool afterwards. 
Section~\ref{sec:zigzags} introduces the concept of zigzags of a graph of an
interval map. 
Moreover, it gives a partial answer to Question~\ref{q:NaQu} and provides some 
interesting examples by applying the results from this section. 
Section~\ref{sec:thin} gives a proof that the permutation technique yields all 
possible thin planar embeddings of chainable continua. Furthermore, we pose some 
related open problems at the end of this section. Finally, 
Section~\ref{sec:nonequivalent} completes the
construction of uncountably many planar embeddings that are not equivalent in the strong sense, 
of every chainable continuum which contains a nondegenerate indecomposable subcontinuum 
and thus answers Question~\ref{q:uncountably} for strong 
equivalence. 
We conclude the paper with some remarks and open questions emerging from the 
study in the final section.

\section{Notation}\label{sec:notation}

Let $\N = \{ 1,2,3,\dots\}$ and $\N_0=\{0,1,2,3,\dots\}$ be the positive and nonnegative integers.
Let $f_i: I=[0,1]\to I$ be continuous surjections for $i\in\N$ and let {\em 
inverse limit space} 
$$
X_{\infty}=\underleftarrow{\lim}\{I, f_i\}=\{(x_0, x_1, x_2, \dots): f_i(x_i)=x_{i-1}, i\in\N\} %\subset I^{\infty}
$$
be equipped with the subspace topology endowed from the product topology of $I^{\infty}$. Let $\pi_i: X_{\infty}\to I$ be the 
{\em coordinate projections} for $i\in\N_0$.

\begin{definition}
Let $X$ be a metric space. A \emph{chain in $X$} is a set 
$\chain=\{\ell_1\ldots, \ell_n\}$ of open subsets of $X$ called \emph{links}, 
such that $\ell_i\cap\ell_j\neq\emptyset$ if and only if $|i-j|\leq 1$. 
If also $\cup_{i=1}^n \ell_i = X$, then we speak of a {\em chain cover} of $X$.
We say that a chain $\chain$ is \emph{nice} if additionally 
all links are open discs (in $X$).

The \emph{mesh of a chain $\chain$} is 
$\mesh(\chain)=\max\{\diam{\ell_i}: i=1, \ldots, n\}$.
A continuum $X$ is {\em chainable} if there exist chain covers of $X$ of arbitrarily 
small mesh. 
\end{definition}

We say that $\chain'=\{\ell'_1, \ldots, \ell'_m\}$ \emph{refines} $\chain$ and 
write $\chain'\preceq\chain$ if for every $j\in\{1, \ldots, m\}$ there exists 
$i\in\{1, \ldots, n\}$ such that $\ell'_j\subset\ell_i$. 
We say that $\chain'$ \emph{properly refines} $\chain$ and 
write $\chain'\prec\chain$ if additionally $\ell'_j\subset\ell_i$ implies that the closure 
$\overline{\ell'_j}\subset\ell_i$.

Let $\chain'\preceq\chain$ be as above. 
The \emph{pattern of $\chain'$ in $\chain$}, denoted by $Pat(\chain', \chain)$,
is the ordered $m$-tuple $(a_1, \ldots, a_m)$ such 
that $\ell'_{j}\subset\ell_{a(j)}$ for every $j\in\{1, \ldots, m\}$ where 
$a(j)\in \{1,\ldots,n \}$. If $\ell'_{j}\subset \ell_{i}\cap\ell_{i+1}$, we take 
$a(j)=i$, but that choice is just by convention.  

For chain $\chain=\{\ell_1, \ldots, \ell_n\}$, write
 $\chain^*=\cup_{i=1}^n\ell_i.$

\section{Construction of natural chains, patterns and nested 
intersections}\label{sec:chains}

First we construct \emph{natural chains} $\chain_n$ for every $n\in\N$ (the terminology originates from \cite{Br2}). 
Take some nice chain cover $C_0=\{l_1^0, \ldots, l_{k(0)}^0\}$ of $I$ and define 
$\chain_0:=\pi_0^{-1}(C_0)=\{\ell_1^0, \ldots, \ell_{k(0)}^0\}$, where 
$\ell_i^0=\pi_0^{-1}(l_i^0)$. Note that $\chain_0$ is a chain cover of 
$X_{\infty}$ (but the links are not necessarily connected sets in 
$X_{\infty}$).  

Now take a nice chain cover $C_1=\{l_1^1, \ldots, l_{k(1)}^1\}$ of $I$ such that for 
every $j\in\{1, \ldots, k(1)\}$ there exists $j'\in\{1, \ldots, k(0)\}$ such 
that $f_1(\overline{l_j^1})\subset l_{j'}^0$ and define 
$\chain_1:=\pi_1^{-1}(C_1)$. Note that $\chain_1$ is a chain cover of 
$X_{\infty}$. Also note that $\chain_1\prec\chain_0$ and $Pat(\chain_1, 
\chain_0)=\{a^1_1, \ldots, a^1_{k(1)}\}$ where 
$f_1(\pi_1(\ell_j^1))\subset\pi_0(\ell_{a^1_j}^0)$ for each $j\in\{1, \ldots, 
k(1)\}$. So the pattern $Pat(\chain_1, \chain_0)$ can easily be calculated by 
just following the graph of $f_1$.

Inductively we construct $\chain_{n+1}=\{\ell^{n+1}_1, \ldots, 
\ell^{n+1}_{k(n+1)}\}:=\pi_{n+1}^{-1}(C_{n+1})$, where $C_{n+1}=\{l^{n+1}_1, 
\ldots, l^{n+1}_{k(n+1)}\}$ is some nice chain cover of $I$ such that for every 
$j\in\{1, \ldots, k(n+1)\}$ there exists $j'\in\{1, \ldots, k(n)\}$ such that 
$f_{n+1}(\overline{l_j^{n+1}})\subset l_{j'}^n$. Note that 
$\chain_{n+1}\prec\chain_n$ and $Pat(\chain_{n+1}, \chain_n)=(a^{n+1}_1, 
\ldots, a^{n+1}_{k(n+1)})$, where 
$f_{n+1}(\pi_{n+1}(\ell_j^{n+1}))\subset\pi_n(\ell_{a^{n+1}_j}^{n})$ for each 
$j\in\{1, \ldots, k(n+1)\}$. 

Throughout the paper we use the straight letter $C$ for chain covers of the interval $I$ and the script letter $\chain$
for chain covers of the inverse limits space.
Note that the links of $C_n$ can be chosen small enough to ensure that 
$\mesh(\chain_n)\to 0$ as $n\to\infty$ 
and note that $X_{\infty}=\cap_{n\in\N_0}\chain_n^*$.

%\section{Patterns and nested intersections}

\begin{lemma}\label{lem:patterns}
	Let $X$ and $Y$ be compact metric spaces and let $\{\chain_n\}_{n \in \N_0}$ and $\{\D_n\}_{n \in \N_0}$ 
be sequences of chains in $X$ and $Y$ respectively such that $\chain_{n+1}\prec\chain_n$, 
$\D_{n+1}\prec\D_n$ and $Pat(\chain_{n+1}, \chain_n)=Pat(\D_{n+1}, \D_n)$ for 
each $n\in\N_0$. Assume also that $\mesh(\chain_n)\to 0$ and $\mesh(\D_n)\to 0$ as 
$n\to\infty$. 
	Then $X'=\cap_{n\in\N_0}\chain_n^*$ and $Y'=\cap_{n\in\N_0}\D_n^*$ are 
nonempty and homeomorphic.
\end{lemma}

\begin{proof}
	To see that $X'$ and $Y'$ are nonempty note that they are nested intersections of nonempty closed sets.
	Define $\chain_k = \{\ell_1^k, \ldots, \ell_{n(k)}^k\}$ and 
$\D_k = \{L_1^k, \ldots, L_{n(k)}^k\}$ for each $k\in\N_0$. Let $x\in X'$. Then 
$x=\cap_{k\in\N_0}\ell_{i(k)}^k$ for some $\ell_{i(k)}^k\in\chain_k$ such that 
$\overline{\ell_{i(k)}^k}\subset\ell_{i(k-1)}^{k-1}$ for each $k\in\N$. Define 
$h: X'\to Y'$ as $h(x):=\cap_{k\in\N_0}L_{i(k)}^k$. Since the patterns agree 
	and diameters tend to zero, $h$ is a well-defined bijection. We 
show that it is continuous. First note that $h(\ell_{i(m)}^m\cap 
X')=L_{i(m)}^m\cap Y'$ for every $m\in\N_0$ and every $i(m)\in\{1, \ldots, 
n(m)\}$, since if $x=\cap_{k\in\N_0}\ell_{i(k)}^k\subset\ell_{i(m)}^m$, then 
there is $k'\in\N_0$ such that $\ell_{i(k)}^k\subset\ell_{i(m)}^m$ for each 
$k\geq k'$. But then $L_{i(k)}^k\subset L_{i(m)}^m$ for each $k\geq k'$, thus 
$h(x)=\cap_{k\in\N_0}L_{i(k)}^k\subset L_{i(m)}^m$. The other direction follows 
analogously. Now let $U\subset Y'$ be an open set and $x\in h^{-1}(U)$. Since 
diameters tend to zero, there is $m\in\N_0$ and $i(m)\in\{1, \ldots, n(m)\}$ 
such that $h(x)\in L_{i(m)}^m\cap Y'\subset U$ and thus $x\in\ell_{i(m)}^m\cap 
X'\subset h^{-1}(U)$. So $h^{-1}(U)\subset X'$ is open and that concludes the 
proof.
\end{proof}

In the following sections we will construct nested intersections of nice planar 
chains such that their patterns are the same as the patterns of refinements 
$\chain_n\prec\chain_{n-1}$ of natural chains of $X_{\infty}$ (as constructed at the beginning of this section) and such that the 
diameters of links tend to zero. By the previous lemma, this gives 
embeddings of $X_{\infty}$ in the plane. We note that the previous lemma holds 
in a more general setting (with an appropriately generalized definition of patterns), \ie for graph-like continua and graph chains, see 
\eg \cite{Medd}. 

\section{Permuting the graph}\label{sec:permuting}
Let $C=\{l_1, \ldots, l_n\}$ be a chain cover of $I$ and let $f: I\to I$ be a 
continuous surjection 
which is piecewise linear with finitely many critical points $0=t_0<t_1< 
\ldots< t_m<t_{m+1}=1$
(so we include the endpoints of $I=[0,1]$ in the set of critical points). 
In the rest of the paper we work with continuous surjections 
which are piecewise linear (so with finitely many critical points);
we call them {\em piecewise linear surjections}.
Without loss of generality we assume that for every $i\in\{0, \ldots, m\}$ and $l \in C$,
$f([t_i, t_{i+1}]) \not\subset l$.

Define $H_{j}=f([t_j, t_{j+1}])\times \{j\}$ for each $j\in\{0, \ldots, m\}$ and 
$V_j=\{f(t_j)\}\times[j-1, j]$ for each $j\in\{1, \ldots, m\}$.
Note that $H_{j-1}$ and $H_{j}$ are joined at their left endpoints by $V_j$ if 
 there is a local minimum of $f$ in $t_{j}$ and they are joined at their right endpoints 
if there is a local maximum of $f$ in $t_j$ (see Figure~\ref{fig:1}). The line $H_0\cup 
V_1\cup H_1\cup\ldots\cup V_m\cup H_m=:G_f$ is called the \emph{flattened graph 
of $f$ in $\R^2$}.

\begin{definition}\label{def:flat}
	A permutation $p:\{0, 1, \ldots, m\}\to\{0, 1, \ldots, m\}$ is called a
\emph{$C$-admissible permutation of $G_f$} if for every $i\in\{0, \ldots, 
m-1\}$ and $k\in\{0, \ldots, m\}$ such that $p(i)<p(k)<p(i+1)$ or 
$p(i+1)<p(k)<p(i)$ it holds that:
	\begin{enumerate}
		\item $f(t_{i+1})\not\in f([t_k, t_{k+1}])$, or 
		\item $f(t_{i+1})\in f([t_k, t_{k+1}])$ but $f(t_k)$ or 
$f(t_{k+1})$ is contained in the same link of $C$ as $f(t_{i+1})$.
	\end{enumerate}
\end{definition}

Denote a $C$-admissible permutation of $G_f$ by 
$$
p^{C}(G_f)=p(H_0)\cup  p(V_1)\cup\ldots\cup p(V_m)\cup p(H_m),
$$ 
for $p(H_j)=f([ \tilde{t}_j, \tilde{t}_{j+1}])\times\{p(j)\}$ and $p(V_j)=\{f(\tilde{t}_j)\}\times[ p(j-1), 
p(j)]$, where $\tilde{t}_j$ is chosen such that $f(t_j)$ and $f(\tilde{t}_j)$ 
are contained in the same link of $C$, and such that $p^C(G_f)$ has no self 
intersections for every $j\in\N$. A line $p^{C}(G_f)$ will be called a 
\emph{permuted graph of $f$ with respect to $C$}. Let $E(p^C(G_f))$ be the endpoint of $p(H_0)$ corresponding 
to $(f(\tilde{t}_0),p(0))$.

Note that $p(V_j)$ from Definition~\ref{def:flat} is a vertical line in the plane which 
joins the endpoints of $p(H_{j-1})$ and $p(H_{j})$ at $f(\tilde{t}_{j})$, see 
Figure~\ref{fig:1}.

\begin{definition}	
	If $p(j)=m$, we say that $H_j$ is \emph{at the top of $p^C(G_f)$}.
\end{definition}

\begin{figure}[!ht]
	\centering
	\begin{tikzpicture}[scale=3]
	
	\draw (0,0)--(0,1)--(1,1)--(1,0)--(0,0);
	\draw (0,0)--(1/4,1)--(0.5,0.3)--(3/4,1)--(1,0.3);
	\draw (1/4,-0.01)--(1/4,0.01);
	\node at (0.27,-0.07) {\scriptsize $t_{1}$};
	\draw (0.5,-0.01)--(0.5,0.01);
	\node at (0.52,-0.07) {\scriptsize $t_{2}$};
	\draw (3/4,-0.01)--(3/4,0.01);
	\node at (0.77,-0.07) {\scriptsize $t_{3}$};
	\node at (0.5,-0.25) {\small $(a)$};
	\draw (-0.01,1/4)--(0.01,1/4);
	\draw (-0.01,1/2)--(0.01,1/2);
	\draw (-0.01,3/4)--(0.01,3/4);
	\node at (-0.05,1/8) {\scriptsize $l_1$};
	\node at (-0.05,3/8) {\scriptsize $l_2$};
	\node at (-0.05,5/8) {\scriptsize $l_3$};
	\node at (-0.05,7/8) {\scriptsize $l_4$};

	\draw (1.5,0)--(2.5,0);
	\draw (1.5, -0.01)--(1.5,0.01);
	\draw (1.75, -0.01)--(1.75,0.01);
	\draw (2, -0.01)--(2,0.01);
	\draw (2.25, -0.01)--(2.25,0.01);
	\draw (2.5, -0.01)--(2.5,0.01);
	\node at (1.625,-0.07) {\scriptsize $l_1$};
	\node at (1.875,-0.07) {\scriptsize $l_2$};
	\node at (2.125,-0.07) {\scriptsize $l_3$};
	\node at (2.375,-0.07) {\scriptsize $l_4$};
	\draw 
(1.5,1/4)--(2.5,1/4)--(2.5,1/2)--(1.85,1/2)--(1.85,3/4)--(2.5,3/4)--(2.5,
1)--(1.85,1);
	\node at (1.45,1/4) {\scriptsize $0$};
	\node at (1.45,1/2) {\scriptsize $1$};
	\node at (1.45,3/4) {\scriptsize $2$};
	\node at (1.45,1) {\scriptsize $3$};
	\node at (2,0.3) {\scriptsize $H_0$};
	\node at (2.25,0.55) {\scriptsize $H_1$};
	\node at (2.25,0.8) {\scriptsize $H_2$};
	\node at (2.25,1.05) {\scriptsize $H_3$};
	\node at (2.58,0.375) {\scriptsize $V_0$};
	\node at (1.77,0.625) {\scriptsize $V_1$};
	\node at (2.58,0.875) {\scriptsize $V_2$};
	\node at (2,-0.25) {\small $(b)$};

	\draw (3.5,0)--(4.5,0);
	\draw (3.5, -0.01)--(3.5,0.01);
	\draw (3.75, -0.01)--(3.75,0.01);
	\draw (4, -0.01)--(4,0.01);
	\draw (4.25, -0.01)--(4.25,0.01);
	\draw (4.5, -0.01)--(4.5,0.01);
	\node at (3.625,-0.07) {\scriptsize $l_1$};
	\node at (3.875,-0.07) {\scriptsize $l_2$};
	\node at (4.125,-0.07) {\scriptsize $l_3$};
	\node at (4.375,-0.07) {\scriptsize $l_4$};
	\draw 
(3.5,1)--(4.5,1)--(4.5,1/4)--(3.85,1/4)--(3.85,0.5)--(4.4,0.5)--(4.4,3/4)--(3.85
,3/4);
	\node at (3.2,1/4) {\scriptsize $0=p(1)$};
	\node at (3.2,1/2) {\scriptsize $1=p(2)$};
	\node at (3.2,3/4) {\scriptsize $2=p(3)$};
	\node at (3.2,1) {\scriptsize $3=p(0)$};
	\node at (4.25,0.3) {\scriptsize $p(H_1)$};
	\node at (4.25,0.55) {\scriptsize $p(H_2)$};
	\node at (4.25,0.8) {\scriptsize $p(H_3)$};
	\node at (4,1.05) {\scriptsize $p(H_0)$};
	\node at (4.65,0.625) {\scriptsize $p(V_0)$};
	\node at (3.7,0.375) {\scriptsize $p(V_1)$};
	\draw[->] (4.4,0.625)--(4.5,0.625);
	\draw[->] (4.5,0.625)--(4.4,0.625);
	\draw[solid, fill] (3.5,1) circle (0.01);
	\node at (3.45,1) {\scriptsize $E$};
	\node at (4,-0.25) {\small $(c)$};
	\end{tikzpicture}
	\caption{Flattened graph and its permutation. Note that $H_0$ is at the 
top of $p^C(G_f)$.}
	\label{fig:1}
\end{figure}
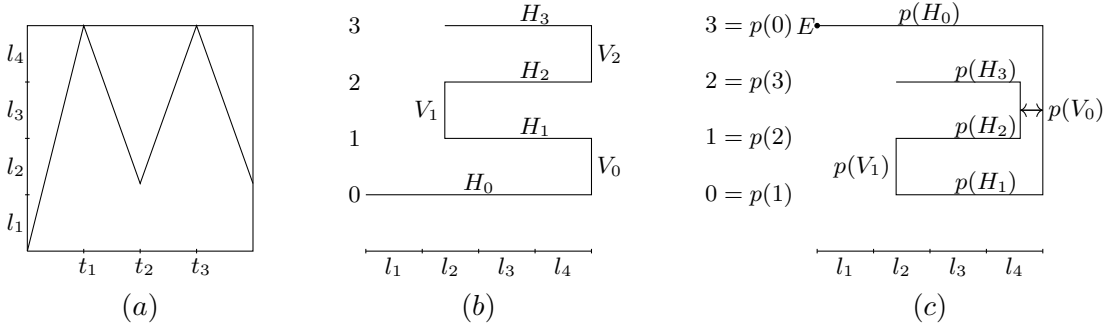

Note that a flattened graph $G_f$ is just a graph of $f$ for which its critical points have 
been extended to vertical intervals. These vertical intervals were introduced for the 
definition of a permuted graph. After permuting the flattened graph, we can quotient 
out the vertical intervals in the following way.

For every $i=1, \ldots, m$, pick a point $q_i\in p(V_i)$. There exists a homotopy 
$F:I\times\R^2\to\R^2$ such that $F(1,y)=q_i$ for every $y\in p(V_i)$ and every 
$i=1, \ldots, m$, and for every $t \in I$, $F(t, \cdot): p^C(G_f)) \to \R^2$ 
is injective, and such that $\pi_x(F(t,(x,y)))=x$ for every $(x,y)\in\R^2$ and $t\in I$. 
Here $\pi_x(x,y)$ denotes a projection on the first coordinate. 
From now on $p^C(G_f)$ will always stand for the quotient $F(1,p^C(G_f))$, but 
for clarity in the figures of Sections~\ref{sec:stretching} and \ref{sec:emb} we will 
continue to draw it with long vertical intervals.

\section{Chain refinements, their composition and 
stretching}\label{sec:stretching}

\begin{definition}
	 Let $f: I\to I$ be a piecewise linear surjection, $p$ an 
admissible $C$-permutation of $G_f$ and $\eps>0$. We call a nice planar chain 
$\chain=\{\ell_1, \ldots, \ell_n\}$ a \emph{tubular $\eps$-chain 
with nerve $p^{C}(G_f)$} if 
	\begin{itemize}
		\item $\chain^*$ is an $\eps$-neighborhood of $p^{C}(G_f)$, and
		\item there exists $n\in\N$ and arcs $A_1\cup \ldots \cup 
A_n=p^{C}(G_f)$ such that $\ell_i$ is the $\eps$-neighborhood of $A_i$ for every 
$i\in\N$.
	\end{itemize}
Denote a nerve $p^{C}(G_f)$ of $\chain$ by $\cN_{\chain}$.
	 When there is no need to specify $\varepsilon$ and 
$\cN_{\chain}$ we just say that $\chain$ is \emph{tubular}. 
\end{definition}

\begin{definition}\label{def:horizontal}
	A planar chain $\chain=\{\ell_1, \ldots, \ell_n\}$ will be called {\em 
horizontal} if there are $\delta>0$ and a chain of open intervals $\{l_1, 
\ldots, l_n\}$ in $\R$ such that $\ell_i=l_i\times(-\delta, \delta)$ for every 
$i\in\{1, \ldots, n\}$.
\end{definition}

\begin{remark}\label{rem:stretch}
	Let $\chain$ be a tubular chain. There exists a homeomorphism 
$\widetilde{H}:\R^2\to\R^2$ 
	such that $\widetilde{H}(\chain)$ is a horizontal chain and 
$\widetilde{H}^{-1}(\chain')$ 
	is tubular for every tubular $\chain'\prec \widetilde{H}(\chain)$.
	Moreover, for $\chain=\{\ell_1, \ldots, \ell_n\}$ denote by 
$\cN_{\widetilde{H}(\chain)}=I\times\{0\}$. Note that 
	$\chain\setminus(\ell_1\cup\ell_n\cup \cN_{\chain})$ has two 
components and thus it makes sense to 
	call the components upper and lower. Let $S$ be the upper component 
of 
	$\chain\setminus(\ell_1\cup\ell_n\cup \cN_{\chain})$.
	
	There exists a homeomorphism $H:\R^2\to\R^2$ which has all the 
properties of a homeomorphism 
	$\widetilde{H}$ above and in addition satisfies:
	\begin{itemize}
		\item the endpoint $H(E(p^{C}(G_f)))=(0,0)$ (recall 
 $E$ from Definition~\ref{def:flat}) and
		\item $H(S)$ is the upper component of 
$H(\chain^*)\setminus(H(\ell_1)\cup H(\ell_n)\cup H(A))$.
	\end{itemize}
	Applying $H$ to a chain $\chain$ is called the \emph{stretching of 
$\chain$} (see Figure~\ref{fig:stretch2}).
\end{remark}

\begin{figure}
	\centering
	\begin{tikzpicture}[scale=2]
	
	\draw[thick] (3,1)--(1,1)--(1,0.5)--(4,0.5)--(4,0)--(2,0);
	
	\draw[thin] 
(3.1,1.1)--(0.9,1.1)--(0.9,0.4)--(3.9,0.4)--(3.9,0.1)--(1.9,0.1)--(1.9,
-0.1)--(4.1,-0.1)--(4.1,0.6)--(1.1,0.6)--(1.1,0.9)--(3.1,0.9)--(3.1,1.1);
	
	\draw[thin] (2.5,0.9)--(2.5,1.1);
	\draw[thin] (1.5,0.9)--(1.5,1.1);
	\draw[thin] (0.9,0.75)--(1.1,0.75);
	\draw[thin] (1.5,0.4)--(1.5,0.6);
	\draw[thin] (2.5,0.4)--(2.5,0.6);
	\draw[thin] (3.5,0.4)--(3.5,0.6);
	\draw[thin] (3.9,0.25)--(4.1,0.25);
	\draw[thin] (3.5,-0.1)--(3.5,0.1);
	\draw[thin] (2.5,-0.1)--(2.5,0.1);
	
	\node at (3.2,1.2) {\small $\chain$};
	
	\draw[solid, fill] (3,1) circle (0.03);
	
	\path[fill=gray, opacity=0.5] 
(3.1,1.1)--(0.9,1.1)--(0.9,0.4)--(3.9,0.4)--(3.9,0.1)--(1.9,0.1)-- 
(1.9,0)--(2,0)--(4,0)--(4,0.5)--(1,0.5)--(1,1)--(3.1,1)--(3.1,1.1);
	\draw[->] (2.5,-0.2)--(2.5,-0.5);
	\node at (2.65,-0.35) {\small $H$};
	\end{tikzpicture}
	\centering
	\begin{tikzpicture}
	\node [draw, dotted, shape=rectangle, minimum width=1cm, minimum 
height=3cm, anchor=center] at (1,0.5) {};
	\node [draw, dotted, shape=rectangle, minimum width=1cm, minimum 
height=3cm, anchor=center] at (2,0.5) {};
	\node [draw, dotted, shape=rectangle, minimum width=1cm, minimum 
height=3cm, anchor=center] at (3,0.5) {};
	\node [draw, dotted, shape=rectangle, minimum width=1cm, minimum 
height=3cm, anchor=center] at (4,0.5) {};
	\node [draw, dotted, shape=rectangle, minimum width=1cm, minimum 
height=3cm, anchor=center] at (5,0.5) {};
	\node [draw, dotted, shape=rectangle, minimum width=1cm, minimum 
height=3cm, anchor=center] at (6,0.5) {};
	\node [draw, dotted, shape=rectangle, minimum width=1cm, minimum 
height=3cm, anchor=center] at (7,0.5) {};
	\node [draw, dotted, shape=rectangle, minimum width=1cm, minimum 
height=3cm, anchor=center] at (8,0.5) {};
	\node [draw, dotted, shape=rectangle, minimum width=1cm, minimum 
height=3cm, anchor=center] at (9,0.5) {};
	\node [draw, dotted, shape=rectangle, minimum width=1cm, minimum 
height=3cm, anchor=center] at (10,0.5) {};
	
	\draw[thick] (1,0.5)--(10,0.5);
	\draw[solid, fill] (1,0.5) circle (0.05);
	\node at (10,1.5) {\small $H(\chain)$};
	
	\path[fill=gray, opacity=0.5] 
(10.5,0.5)--(10.5,2)--(0.5,2)--(0.5,0.5)--cycle;
	\end{tikzpicture}
	\caption{Stretching the chain $\chain$. Recall that the vertical intervals are actually identified with points and thus $\tilde{H}^{-1}(\chain')$ is tubular for 
	every tubular $\chain'\prec\tilde{H}(\chain)$.}
	\label{fig:stretch2}
\end{figure}
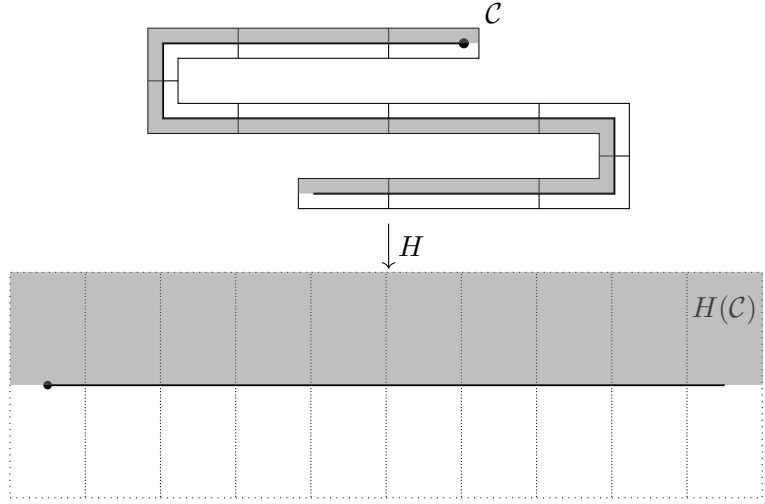

\begin{remark}\label{rem:ref}
Let $X_{\infty}, \{C_i\}_{i\in\N_0}, \{\chain_i\}_{i\in\N_0}$ be as defined 
in Section~\ref{sec:chains}. For $i\in\N_0$, let $\D_i$ be a horizontal chain 
with the same number of links as $\chain_i$ and such that 
$p^{C_i}(G_{f_{i+1}})\subset\D_i^*$ for some $C_i$-admissible permutation $p$. 
Fix $\eps>0$ and note that, after possibly dividing links of $\chain_{i+1}$ into smaller links 
(\ie refining the chain $C_{i+1}$ of $I$), there exists a tubular chain 
	$\D_{i+1}\prec\D_i$ with nerve $p^{C_i}(G_{f_{i+1}})$ such that 
	$Pat(\D_{i+1}, \D_i)=Pat(\chain_{i+1}, \chain_i)$ and $\mesh(\D_{i+1})<\eps$, see 
Figure~\ref{fig:reftube}.
\end{remark}

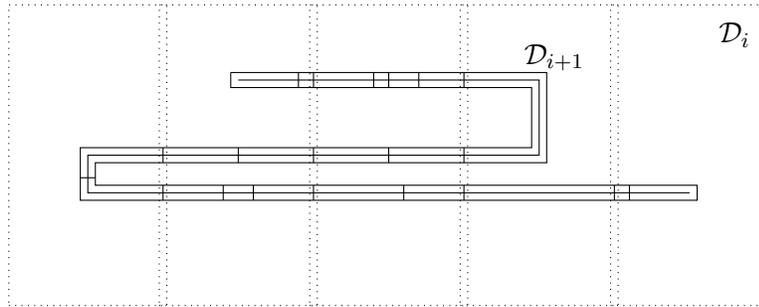
\begin{figure}[!ht]
	\centering
	\begin{tikzpicture}[scale=2]
	\node [draw, dotted, shape=rectangle, minimum width=2.1cm, minimum 
height=4cm, anchor=center] at (1,0.5) {};
	\node [draw, dotted, shape=rectangle, minimum width=2.1cm, minimum 
height=4cm, anchor=center] at (2,0.5) {};
	\node [draw, dotted, shape=rectangle, minimum width=2.1cm, minimum 
height=4cm, anchor=center] at (3,0.5) {};
	\node [draw, dotted, shape=rectangle, minimum width=2.1cm, minimum 
height=4cm, anchor=center] at (4,0.5) {};
	\node [draw, dotted, shape=rectangle, minimum width=2.1cm, minimum 
height=4cm, anchor=center] at (5,0.5) {};
	
	\draw (2,1)--(4,1)--(4,0.5)--(1,0.5)--(1,0.25)--(5,0.25);
	\draw[thin] 
(1.95,1.05)--(4.05,1.05)--(4.05,0.45)--(1.05,0.45)--(1.05,0.3)--(5.05,
0.3)--(5.05,0.2)--(0.95,0.2)--(0.95,0.55)--(3.95,0.55)--(3.95,0.95)--(1.95,
0.95)--(1.95,1.05);
	\draw[thin] (2.4,0.95)--(2.4,1.05);
	\draw[thin] (2.5,0.95)--(2.5,1.05);
	\draw[thin] (2.9,0.95)--(2.9,1.05);
	\draw[thin] (3,0.95)--(3,1.05);
	\draw[thin] (3.2,0.95)--(3.2,1.05);
	\draw[thin] (3.5,0.95)--(3.5,1.05);
	\draw[thin] (3.5,0.45)--(3.5,0.55);
	\draw[thin] (3,0.45)--(3,0.55);
	\draw[thin] (2,0.45)--(2,0.55);
	\draw[thin] (2.5,0.45)--(2.5,0.55);
	\draw[thin] (1.5,0.45)--(1.5,0.55);
	\draw[thin] (0.95,0.35)--(1.05,0.35);
	\draw[thin] (1.5,0.2)--(1.5,0.3);
	\draw[thin] (1.9,0.2)--(1.9,0.3);
	\draw[thin] (2.1,0.2)--(2.1,0.3);
	\draw[thin] (2.5,0.2)--(2.5,0.3);
	\draw[thin] (3.1,0.2)--(3.1,0.3);
	\draw[thin] (3.5,0.2)--(3.5,0.3);
	\draw[thin] (4.5,0.2)--(4.5,0.3);
	\draw[thin] (4.6,0.2)--(4.6,0.3);
	
	\node at (5.3,1.3) {\small $\D_i$};
	\node at (4.1,1.15) {\small $\D_{i+1}$};
	\end{tikzpicture}
	\caption{Constructing a tubular chain with nerve 
$p^{C}(G_f)$. Recall that vertical intervals represent points.}
	\label{fig:reftube}
\end{figure}

\begin{definition}
	Let $H:\R^2\to\R^2$ be a stretching of some tubular chain $\chain$. If 
$\chain'$ is a nice chain in $\R^2$ refining $\chain$ and there is an interval 
map $g: I\to I$ such that $p^{C}(G_g)$ is a nerve of $H(\chain')$, then we say 
that \emph{$\chain'$ follows $p^{C}(G_g)$ in $\chain$}.
\end{definition}

Now we discuss compositions of chain refinements. Let $f, g: I\to I$ be 
piecewise linear surjections. Let $0=t_0<t_1< \ldots< t_{m}<t_{m+1}=1$ be 
the critical points of $f$ and let $0=s_0<s_1< \ldots< s_{n}<s_{n+1}=1$ be
the critical points of $g$. Let $C_1$ and $C_2$ be nice chain covers of $I$, let 
$p_1:\{0, 1, \ldots, m\}\to\{0, 1, \ldots, m\}$ 
be an admissible $C_1$-permutation of $G_f$ and let $p_2:\{0, 1, \ldots, 
n\}\to\{0, 1, \ldots, n\}$ be an admissible 
$C_2$-permutation of $G_g$. 

Assume $\chain''\prec\chain'\prec\chain$ are nice chains in $\R^2$ such that 
$\chain$ is horizontal 
and $p_1^{C_1}(G_f)\subset\chain^*$
(recall that $\chain^*$ denotes the union of the links of $\chain$), 
$\chain'$ is a tubular chain with $\cN_{\chain'}=p_1^{C_1}(G_f)$, and 
$\chain''$ follows $p_2^{C_2}(G_g)$ in $\chain'$. Then $\chain''$ follows 
$f\circ g$ in $\chain$ with respect to a $C_1$-admissible permutation of 
$G_{f\circ g}$ which we will denote by $p_1*p_2$ (see Figures~\ref{fig:3} and 
\ref{fig:4}).

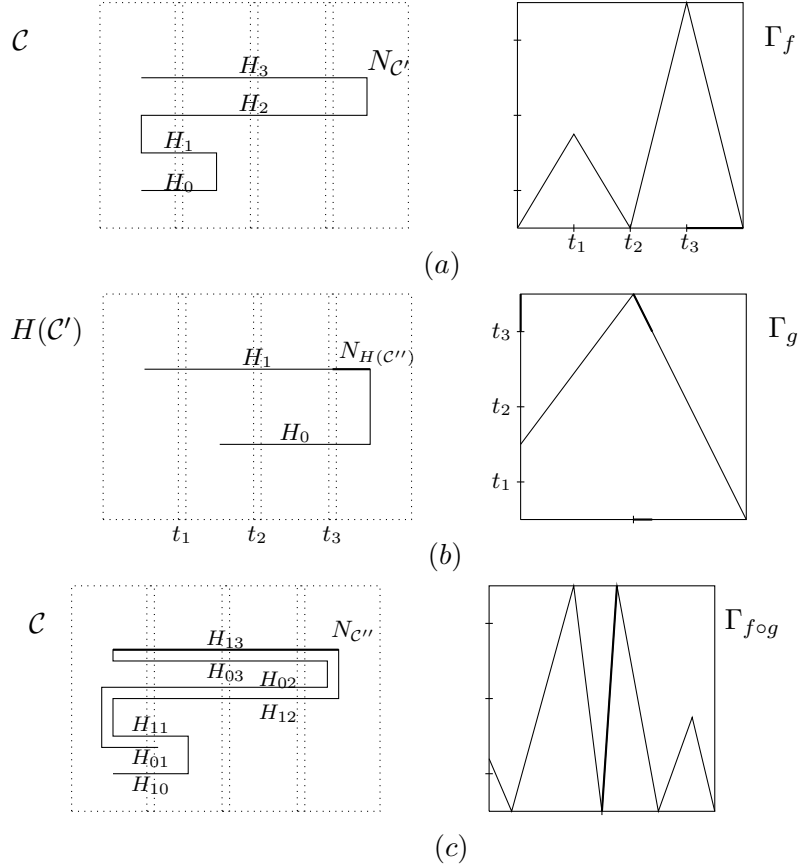
\begin{figure}[!ht]
	\centering
	\begin{tikzpicture}[scale=1]
	
	\node [draw, dotted, shape=rectangle, minimum width=1.1cm, minimum 
height=3cm, anchor=center] at (1,1) {};
	\node [draw, dotted, shape=rectangle, minimum width=1.1cm, minimum 
height=3cm, anchor=center] at (2,1) {};
	\node [draw, dotted, shape=rectangle, minimum width=1.1cm, minimum 
height=3cm, anchor=center] at (3,1) {};
	\node [draw, dotted, shape=rectangle, minimum width=1.1cm, minimum 
height=3cm, anchor=center] at (4,1) {};
	
	\draw (1,0)--(2,0)--(2,0.5)--(1,0.5)--(1,1)--(4,1)--(4,1.5)--(1,1.5);
	\node at (1.5,0.1) {\scriptsize $H_0$};
	\node at (1.5,0.65) {\scriptsize $H_1$};
	\node at (2.5,1.15) {\scriptsize $H_2$};
	\node at (2.5,1.65) {\scriptsize $H_3$};
	\node at (-0.6,2) {\small $\chain$};
	\node at (4.3,1.7) {\small $N_{\chain'}$};
	
	\draw (6,-0.5)--(6,2.5)--(9,2.5)--(9,-0.5)--(6,-0.5);
	\draw (6,-0.5)--(6.75,0.75)--(7.5,-0.5)--(8.25,2.5)--(9,-0.5);
	\draw (5.95,0)--(6.05,0);
	\draw (5.95,1)--(6.05,1);
	\draw (5.95,2)--(6.05,2);
	\draw (6.75,-0.45)--(6.75,-0.55);
	\draw (7.5,-0.45)--(7.5,-0.55);
	\draw[thick] (8.25,-0.5)--(9,-0.5);
	\draw (8.25,-0.45)--(8.25,-0.55);
	\node at (6.8,-0.7) {\scriptsize $t_1$};
	\node at (8.3,-0.7) {\scriptsize $t_3$};
	\node at (7.55,-0.7) {\scriptsize $t_2$};
	\node at (9.5,2) {\small $\Gamma_f$};
	\node at (5,-1) {\small $(a)$};
	
	\end{tikzpicture}
	\begin{tikzpicture}[scale=1]
	
	\node [draw, dotted, shape=rectangle, minimum width=1.1cm, minimum 
height=3cm, anchor=center] at (1,1) {};
	\node [draw, dotted, shape=rectangle, minimum width=1.1cm, minimum 
height=3cm, anchor=center] at (2,1) {};
	\node [draw, dotted, shape=rectangle, minimum width=1.1cm, minimum 
height=3cm, anchor=center] at (3,1) {};
	\node [draw, dotted, shape=rectangle, minimum width=1.1cm, minimum 
height=3cm, anchor=center] at (4,1) {};
	
	\draw (2,0.5)--(4,0.5)--(4,1.5)--(1,1.5);
	\node at (3,0.65) {\scriptsize $H_0$};
	\node at (2.5,1.65) {\scriptsize $H_1$};
	
	\node at (-0.3,2) {\small $H(\chain')$};
	\node at (4.1,1.7) {\scriptsize $N_{H(\chain'')}$};
	
	\draw (6,-0.5)--(6,2.5)--(9,2.5)--(9,-0.5)--(6,-0.5);
	\draw (6,0.5)--(7.5,2.5)--(9,-0.5);
	\draw (5.95,0)--(6.05,0);
	\draw (5.95,1)--(6.05,1);
	\draw (5.95,2)--(6.05,2);
	\draw[thick] (6,2)--(6,2.5);
	\draw[thick] (7.5,-0.5)--(7.75,-0.5);
	\draw[thick] (7.5,2.5)--(7.75,2);
	\draw (7.5,-0.45)--(7.5,-0.55);
	\node at (5.75,0) {\scriptsize $t_1$};
	\node at (5.75,1) {\scriptsize $t_2$};
	\node at (5.75,2) {\scriptsize $t_3$};
	\node at (9.5,2) {\small $\Gamma_g$};
	
	\draw[thick] (3.5,1.5)--(4,1.5);
	\node at (5,-1) {\small $(b)$};
	
	\node at (1.5,-0.7) {\scriptsize $t_1$};
	\node at (2.5,-0.7) {\scriptsize $t_2$};
	\node at (3.5,-0.7) {\scriptsize $t_3$};
	%\draw (0.5,0)--(4.5,0);
	
	\end{tikzpicture}
	\begin{tikzpicture}[scale=1]
	
	\node [draw, dotted, shape=rectangle, minimum width=1.1cm, minimum 
height=3cm, anchor=center] at (1,1) {};
	\node [draw, dotted, shape=rectangle, minimum width=1.1cm, minimum 
height=3cm, anchor=center] at (2,1) {};
	\node [draw, dotted, shape=rectangle, minimum width=1.1cm, minimum 
height=3cm, anchor=center] at (3,1) {};
	\node [draw, dotted, shape=rectangle, minimum width=1.1cm, minimum 
height=3cm, anchor=center] at (4,1) {};
	
	\draw 
(1,0)--(2,0)--(2,0.5)--(1,0.5)--(1,1)--(4,1)--(4,1.65)--(1,1.65)--(1,1.5)--(3.85
,1.5)--(3.85,1.15)--(0.85,1.15)--(0.85,0.35)--(1.6,0.35);
	
	\draw[thick] (4,1.65)--(1,1.65);
	\node at (2.5,1.79) {\tiny $H_{13}$};
	\node at (2.5,1.35) {\tiny $H_{03}$};
	\node at (3.2,1.25) {\tiny $H_{02}$};
	\node at (3.2,0.8) {\tiny $H_{12}$};
	\node at (1.5,0.65) {\tiny $H_{11}$};
	\node at (1.5,0.2) {\tiny $H_{01}$};
	\node at (1.5,-0.15) {\tiny $H_{10}$};
	
	\draw (6,-0.5)--(6,2.5)--(9,2.5)--(9,-0.5)--(6,-0.5);
	\draw 
(6,0.2)--(6.3,-0.5)--(28.5/4,2.5)--(7.5,-0.5)--(7.7,2.5)--(8.25,-0.5)--(8.7,
0.75)--(9,-0.5);
	\draw (5.95,0)--(6.05,0);
	\draw (5.95,1)--(6.05,1);
	\draw (5.95,2)--(6.05,2);
	\draw (7.5,-0.45)--(7.5,-0.55);
	\draw[thick] (7.5,-0.5)--(7.7,2.5);
	\node at (9.5,2) {\small $\Gamma_{f\circ g}$};
	
	\node at (0,2) {\small $\chain$};
	\node at (4.2,1.9) {\scriptsize $N_{\chain''}$};
	\node at (5.5,-1) {\small $(c)$};
	\end{tikzpicture}
	\caption{Composing refinements. In $(a)$ the horizontal chain $\chain$ 
and a nerve of $\chain'$ are drawn. Nerve $\cN_{\chain'}$ equals 
$G_{f}^{C_1}$, a flattened version of the graph $\Gamma_f$. In $(b)$ we draw 
$\chain'$ as a horizontal chain by applying $H$. Also, nerve 
$\cN_{H(\chain'')}$ is given as $G_{g}^{C_2}$, a flattened version of the graph 
$\Gamma_g$. 
	In $(c)$ we draw  $N_{\chain''}$ in $\chain$. In bold we trace the arc 
which is the top of $(id*id)^{C_1}(G_{f\circ g})=N_{\chain''}$.}
	\label{fig:3}
\end{figure}

Define
$$
A_{ij}=\{x\in I: x\in[s_i, s_{i+1}], g(x)\in[t_j, t_{j+1}]\},
$$
for $i\in\{0, 1, \ldots, n\}, j\in\{0, 1, \ldots, m\}$, \ie $A_{ij}$ are 
maximal intervals on which $f\circ g$ 
is injective and possibly $A_{ij} = \emptyset$.
Let $H_{ij}$ be the horizontal branch of $G_{f\circ g}$ corresponding to 
the interval $A_{ij}$.

We want to see which branch $H_{ij}$ corresponds to the top of 
$(p_1*p_2)^{C_1}(G_{f\circ g})$. 
Denote the top of $p_1^{C_1}(G_f)$ by $p_1(H_{T_1})$, \ie $p_1(T_1)=m$. Denote 
the top of $p_2^{C_2}(G_g)$ by $p_2(H_{T_2})$, \ie $p_2(T_2)=n$. By the choice 
of orientation of $H$, the top of $(p_1*p_2)^{C_1}(G_{f\circ g})$ is 
$H_{T_2T_1}$ (see Figures~\ref{fig:3} and \ref{fig:4}).

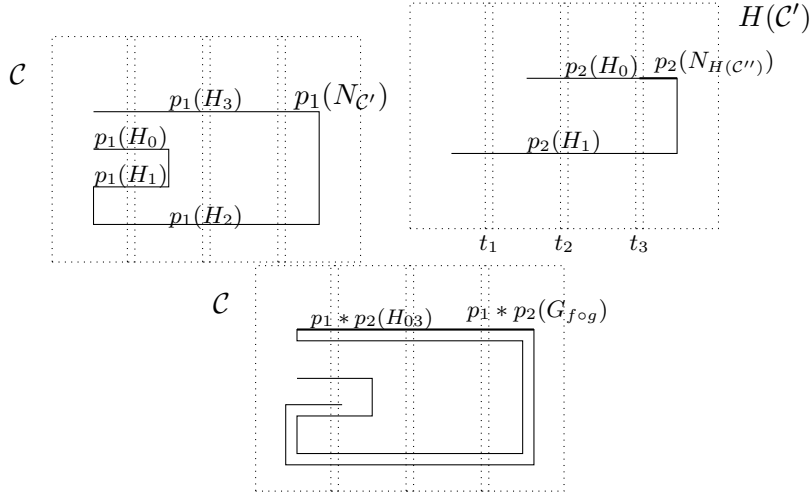
\begin{figure}[!ht]
	\centering
	\begin{tikzpicture}[scale=1]
	
	\node [draw, dotted, shape=rectangle, minimum width=1.1cm, minimum 
height=3cm, anchor=center] at (1,1) {};
	\node [draw, dotted, shape=rectangle, minimum width=1.1cm, minimum 
height=3cm, anchor=center] at (2,1) {};
	\node [draw, dotted, shape=rectangle, minimum width=1.1cm, minimum 
height=3cm, anchor=center] at (3,1) {};
	\node [draw, dotted, shape=rectangle, minimum width=1.1cm, minimum 
height=3cm, anchor=center] at (4,1) {};
	
	\draw (1,1)--(2,1)--(2,0.5)--(1,0.5)--(1,0)--(4,0)--(4,1.5)--(1,1.5);
	\node at (2.5,0.1) {\scriptsize $p_1(H_2)$};
	\node at (1.5,0.65) {\scriptsize $p_1(H_1)$};
	\node at (1.5,1.15) {\scriptsize $p_1(H_0)$};
	\node at (2.5,1.65) {\scriptsize $p_1(H_3)$};
	\node at (0,2) {\small $\chain$};
	\node at (4.3,1.7) {\small $p_1(N_{\chain'})$};
	
	\end{tikzpicture}
	\begin{tikzpicture}[scale=1]
	
	\node [draw, dotted, shape=rectangle, minimum width=1.1cm, minimum 
height=3cm, anchor=center] at (1,1) {};
	\node [draw, dotted, shape=rectangle, minimum width=1.1cm, minimum 
height=3cm, anchor=center] at (2,1) {};
	\node [draw, dotted, shape=rectangle, minimum width=1.1cm, minimum 
height=3cm, anchor=center] at (3,1) {};
	\node [draw, dotted, shape=rectangle, minimum width=1.1cm, minimum 
height=3cm, anchor=center] at (4,1) {};
	
	\draw (1,0.5)--(4,0.5)--(4,1.5)--(2,1.5);
	\node at (2.5,0.65) {\scriptsize $p_2(H_1)$};
	\node at (3,1.65) {\scriptsize $p_2(H_0)$};
	
	\node at (5.3,2.3) {\small $H(\chain')$};
	\node at (4.5,1.7) {\scriptsize $p_2(N_{H(\chain'')})$};

	\draw[thick] (3.5,1.5)--(4,1.5);

	\node at (1.5,-0.7) {\scriptsize $t_1$};
	\node at (2.5,-0.7) {\scriptsize $t_2$};
	\node at (3.5,-0.7) {\scriptsize $t_3$};
	%\draw (0.5,0)--(4.5,0);
	\end{tikzpicture}
	\begin{tikzpicture}[scale=1]
	
	\node [draw, dotted, shape=rectangle, minimum width=1.1cm, minimum 
height=3cm, anchor=center] at (1,1) {};
	\node [draw, dotted, shape=rectangle, minimum width=1.1cm, minimum 
height=3cm, anchor=center] at (2,1) {};
	\node [draw, dotted, shape=rectangle, minimum width=1.1cm, minimum 
height=3cm, anchor=center] at (3,1) {};
	\node [draw, dotted, shape=rectangle, minimum width=1.1cm, minimum 
height=3cm, anchor=center] at (4,1) {};
	
	\draw 
(1,1)--(2,1)--(2,0.5)--(1,0.5)--(1,0)--(4,0)--(4,1.5)--(1,1.5)--(1,1.65)--(4.15,
1.65)--(4.15,-0.15)--(0.85,-0.15)--(0.85, 0.65)--(1.6,0.65);
	\draw[thick] (1,1.65)--(4.15,1.65);
	
	\node at (0,2) {\small $\chain$};
	\node at (4.2,1.9) {\scriptsize $p_1*p_2(G_{f\circ g})$};
	\node at (2,1.8) {\tiny $p_1*p_2(H_{03})$};
	
	\end{tikzpicture}
	\caption{Composing permuted refinements. Here $p_1=(0\ 2)$ and $p_2=(0\ 
1)$ are admissible. The top of $p_1(N_{\chain'})$ is $p_1(H_3)$, so $T_1=3$. 
The top of $p_2(N_{H(\chain'')})$ is $p_2(H_0)$, so $T_2=0$. Thus, the top of 
$(p_1*p_2)^{C_1}(G_{f\circ g})$ is $H_{T_2T_1}=H_{03}$ (in bold).}
	\label{fig:4}
\end{figure}

\section{Construction of the embeddings}\label{sec:emb}

Let $X_{\infty}=\underleftarrow{\lim}\{I, f_i\}$ where for every $i\in\N$ the 
map $f_i$ is a continuous  piecewise linear surjection with 
critical points $0=t_0^i<t_1^i< \ldots< t_{m(i)}^i<t_{m(i)+1}^i=1$. 
Let $I_k^i=[t_k^i, t_{k+1}^i]$ for every $i\in\N$ and every $k\in\{0, \ldots, 
m(i)\}$.
We construct chains $(C_n)_{n\in\N_0}$ and $(\chain_n)_{n\in\N_0}$ as before, 
such that for each $i \in \N_0$, $k\in\{0, \ldots, m(i+1)\}$ and $l \in C_i$,
$f_{i+1}(I_{k}^{i+1}) \not\subset l$. The flattened graph of 
$f_{i}$ will be denoted by $G_{f_{i}}=H^i_0\cup V^i_1\cup\ldots\cup 
V^i_{m(i)}\cup H^i_{m(i)}$ for each $i\in\N_0$.

\begin{theorem}\label{thm:algorithm}
Let $x=(x_0, x_1, x_2, \dots)\in X_{\infty}$ be such that
for each $i\in\N_0$,  $x_i\in 
I_{k(i)}^i$  and there exists an admissible permutation  
	(with respect to $C_{i-1}$) $p_i:\{0, \ldots, m(i)\}\to\{0, \ldots, 
m(i)\}$ of $G_{f_{i}}$ 
	such that $p_i(k(i))=m(i)$. Then there exists a planar embedding of 
$X_{\infty}$ such that $x$ is accessible.
\end{theorem}
\begin{proof}
	Fix a strictly decreasing sequence $(\eps_i)_{i\in\N}$ such that 
$\eps_i\to 0$ as $i\to\infty$.
	Let $\D_0$ be a nice horizontal chain in $\R^2$ with the same number of 
links as $\chain_0$. 
	By Remark~\ref{rem:ref} we can find a tubular chain 
$\D_1\prec\D_0$ with nerve $p_1^{C_0}(G_{f_1})$, such that $Pat(\D_1, 
\D_0)=Pat(\chain_1, \chain_0)$ and $\mesh(\D_1)<\eps_1$. Note that $p_1(k(1))=m(1)$.
	
	Let $F:\R^2\to\R^2$ be a stretching of $\D_1$ (see 
Remark~\ref{rem:stretch}). Again using Remark~\ref{rem:ref} we can define $F(\D_2)\prec F(\D_1)$ such that 
$\mesh(\D_2)<\eps_2$ ($F$ is uniformly continuous), $Pat(F(\D_2), 
F(\D_1))=Pat(\chain_2, \chain_1)$ and nerve of $F(\D_2)$ is 
$p_2^{C_1}(G_{f_2})$. Thus $H^2_{k(2)}$ is the top of $N_{F(\D_2)}$. By the 
arguments in the previous section, the top of $N_{\D_2}$ is $H_{k(2)k(1)}$.
	
	As in the previous section, denote the maximal intervals of 
monotonicity of $f_1\circ\ldots\circ f_i$ by
	$$A_{n(i)\ldots n(1)}:=\{x\in I: x\in I_{n(i)}^i, f_i(x)\in 
I_{n(i-1)}^{i-1}, \ldots, f_1\circ\ldots\circ f_{i-1}(x)\in I_{n(1)}^1\},$$
	and denote the corresponding horizontal intervals of 
$G_{f_1\circ\ldots\circ f_i}$ by $H_{n(i)\ldots n(1)}$.
	
	Assume that we have constructed a sequence of chains 
$\D_i\prec\D_{i-1}\prec\ldots\prec\D_1\prec\D_0$. Take a stretching 
$F:\R^2\to\R^2$ of $\D_i$ and define $F(\D_{i+1})\prec F(\D_i)$ such that 
$\mesh(\D_{i+1})<\eps_{i+1}$, $Pat(F(\D_{i+1}), F(\D_i))=Pat(\chain_{i+1}, 
\chain_i)$ and such that a nerve of $F(\D_{i+1})$ is 
$p_{i+1}^{C_i}(G_{f_{i+1}})$, which is possible by Remark~\ref{rem:ref}. Note 
that the top of ${\mathcal N}_{\D_{i+1}}$ is $H_{k(i+1)\ldots k(1)}$.
	
	Since $Pat(F(\D_{i+1}), F(\D_i))=Pat(\D_{i+1}, \D_i)$ 
	for every $i\in\N_0$ and by the choice of the sequence $(\eps_i)$, 
Lemma~\ref{lem:patterns} yields that
	$\cap_{n\in\N_0}\D_n^*$ is homeomorphic to $X_{\infty}$. 
	Let $\phi(X_{\infty})=\cap_{n\in\N_0}\D_n^*$.
	
	To see that $x$ is accessible, note that 
$H=\lim_{i\to\infty}H_{k(i)\ldots k(1)}$ is a well-defined horizontal arc 
	in $\phi(X_{\infty})$ (possibly degenerate). Let $H=[a, b]\times\{h\}$ 
for some $h\in\R$. Note that for every $y=(y_1, y_2)\in\phi(X_{\infty})$ it 
holds that $y_2\leq h$. Thus every point $p=(p_1, h)\in H$ is accessible by the 
vertical planar arc $\{p_1\}\times[ h, h+1]$. Since $x\in H$, the construction 
is complete. 
\end{proof}

\section{Zigzags}\label{sec:zigzags}
\begin{definition}\label{def:zigzag}
	Let $f: I\to I$ be a continuous piecewise monotone surjection with 
critical points $0=t_0<t_1< \ldots< t_{m}<t_{m+1}=1$. Let
$I_k=[t_k, t_{k+1}]$ for every $k\in\{0, \ldots, m\}$. We say that $I_k$ is 
\emph{inside a zigzag of $f$} if there exist critical points $a$ and $e$ of $f$ 
such that $a<t_k<t_{k+1}<e\in I$ and either
\begin{enumerate}
		\item $f(t_k)>f(t_{k+1})$ and $f|_{[a,e]}$ assumes its global maximum at $a$ and its global minimum
at $e$, or
		\item $f(t_k)<f(t_{k+1})$ and $f|_{[a,e]}$ assumes its global minimum at $a$ and its 
		global maximum at $e$.
	\end{enumerate}
Then we say that $x\in \mathring{I}_k=I_k\setminus\{t_k,t_{k+1}\}$ is {\em inside a zigzag of $f$}
(see Figure~\ref{fig:zigzag}).
We also say that \emph{$f$ contains a zigzag} if there is a point inside a zigzag of $f$.
\end{definition}

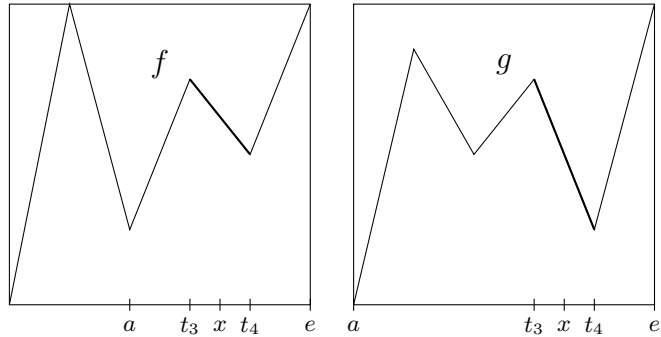
\begin{figure}[!ht]
	\centering
	\begin{tikzpicture}[scale=4]
	\draw (0,0)--(1,0)--(1,1)--(0,1)--(0,0);
	\draw (0,0)--(1/5,1)--(2/5,1/4)--(3/5,3/4)--(4/5,1/2)--(1,1);
	\draw (2/5,-0.02)--(2/5,0.02);
	\node at (0.5, 0.8) { $f$};
	\node at (2/5, -0.07) {\scriptsize $a$};
	\draw (3/5,-0.02)--(3/5,0.02);
	\node at (3/5, -0.07) {\scriptsize $t_3$};
	\draw (7/10,-0.02)--(7/10,0.02);
	\node at (7/10, -0.07) {\scriptsize $x$};
	\draw (4/5,-0.02)--(4/5,0.02);
	\node at (4/5, -0.07) {\scriptsize $t_4$};
	\draw (1,-0.02)--(1,0.02);
	\node at (1, -0.07) {\scriptsize $e$};
	\draw[thick] (3/5,3/4)--(4/5,1/2);
	\end{tikzpicture}
	\begin{tikzpicture}[scale=4]
	\draw (0,0)--(1,0)--(1,1)--(0,1)--(0,0);
	\draw (0,0)--(0.2,0.85)--(0.4,0.5)--(0.6,0.75)--(0.8,0.25)--(1,1);
	\draw (0,-0.02)--(0,0.02);
	\node at (0.5, 0.8) { $g$};
	\node at (0, -0.07) {\scriptsize $a$};
	\draw (3/5,-0.02)--(3/5,0.02);
	\node at (3/5, -0.07) {\scriptsize $t_3$};
	\draw (7/10,-0.02)--(7/10,0.02);
	\node at (7/10, -0.07) {\scriptsize $x$};
	\draw (4/5,-0.02)--(4/5,0.02);
	\node at (4/5, -0.07) {\scriptsize $t_4$};
	\draw (1,-0.02)--(1,0.02);
	\node at (1, -0.07) {\scriptsize $e$};
	\draw[thick] (0.6,0.75)--(0.8,0.25);
	\end{tikzpicture}
	\caption{The interval $[t_3,t_4]$ is inside a zigzag of $f$ and $g$.}
	\label{fig:zigzag}
\end{figure}

\begin{lemma}\label{lem:zigzag}
	Let $f: I\to I$ be a continuous piecewise linear surjection with 
 critical points $0=t_0<t_1<\ldots <t_m<t_{m+1}=1$. If there is
$k\in\{0, \ldots, m\}$ such that $I_k=[t_k, 
t_{k+1}]$ is not inside a zigzag of $f$, then there 
exists an admissible permutation $p$ of $G_f$ (with respect to any nice chain 
$C$) such that $p(k)=m$.
\end{lemma}

\begin{proof}
Assume that $I_k$ is not inside a zigzag of $f$. Without loss of 
generality assume that $f(t_k)>f(t_{k+1})$. If $f(a)\geq f(t_{k+1})$ for each $a\in[0, 
t_k]$ (or if $f(e)\leq f(t_k)$ for each $e\in[t_{k+1}, 1]$) we are done
(see Figure~\ref{fig:refl1}) by simply 
reflecting all $H_i$, $i<k$ over $H_k$ (or reflecting all $H_i$, $i>k$ over $H_k$ in 
the second case).

\begin{figure}[!ht]
	\centering
	\begin{tikzpicture}[scale=2]
	\draw 
(0,1)--(0,2.5)--(0.2,2.5)--(0.2,2)--(0.4,2)--(0.4,3)--(0.6,3)--(0.6,1.5)--(0.8,
1.5)--(0.8,2.3)--(1,2.3)--(1,0.5)--(1.2,0.5)--(1.2,3.5);
	\draw[thick] (1,2.3)--(1,0.5);
	\draw[->] (1.5,2)--(2,2);
	\draw (2.5,0.5)--(3.7,0.5)--(3.7,3.5);
	\draw[thick] (2.5,2.3)--(2.5,0.5);
	\draw[dashed] 
(3.5,1)--(3.5,2.5)--(3.3,2.5)--(3.3,2)--(3.1,2)--(3.1,3)--(2.9,3)--(2.9,
1.5)--(2.7,1.5)--(2.7,2.3)--(2.5,2.3);
	\node at (0.95,2.4) {\small $f(t_k)$};
	\node at (1.2,0.38) {\small $f(t_{k+1})$};
	\end{tikzpicture}
	\caption{Reflections in the first part of the proof of 
Lemma~\ref{lem:zigzag}.}
	\label{fig:refl1}
\end{figure}
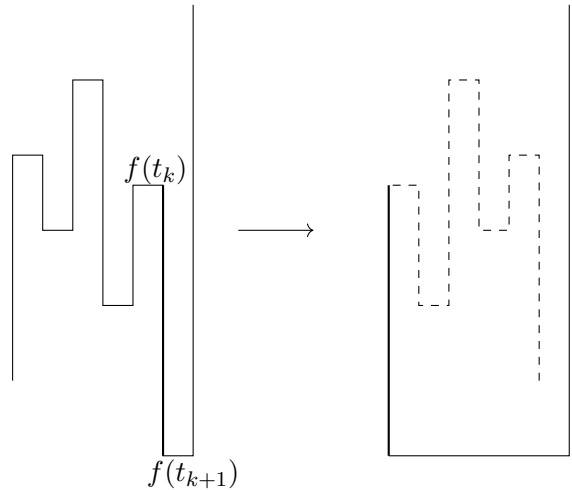

	Therefore, assume that there exists $a\in[0, t_k]$ such that 
$f(a)<f(t_{k+1})$ and there exists $e\in[t_{k+1},1]$ such that $f(e)>f(t_k)$. 
Denote the largest such $a$ by $a_1$ and the smallest such $e$ by $e_1$. Since 
$I_k$ is not inside a zigzag, there exists $e'\in[t_{k+1}, e_1]$ such that 
$f(e')\leq f(a_1)$ or there exists $a'\in[a_1, t_k]$ such that $f(a')\geq 
f(e_1)$. 
	Assume the first case and take $e'$ for which 
$f|_{[t_{k+1}, e_1]}$ attains its global minimum
	(in the second case we take $a'$ for which $f|_{[a_1, t_k]}$ attains its global maximum). Reflect $f|_{[a_1, t_k]}$ over $f|_{[t_k, e']}$ (in the 
second case we reflect $f|_{[t_{k+1}, e_1]}$ over $f|_{[a', t_{k+1}]}$). Then, 
$H_k$ becomes the top of $G_{f|_{[a_1, e_1]}}$ (see Figure~\ref{fig:refl2}).

\begin{figure}[!ht]
	\centering
	\begin{tikzpicture}[scale=2]
	\draw 
(0,0.5)--(0,2)--(0.2,2)--(0.2,1)--(0.4,1)--(0.4,1.5)--(0.6,1.5)--(0.6,0)--(0.8,
0)--(0.8,2.5);
	\draw[thick] (0.2,2)--(0.2,1);
	\node at (0.02,0.4) {\small $a_1$};
	\node at (0.1,2.1) {\small $t_k$};
	\node at (0.3,0.9) {\small $t_{k+1}$};
	\node at (0.7,-0.1) {\small $e'$};
	\node at (0.8,2.55) {\small $e_1$};
	
	\draw[->] (1.3,1.25)--(1.8,1.25);
	\draw(2.5,2)--(2.5,1)-- 
(2.7,1)--(2.7,1.5)--(2.9,1.5)--(2.9,0)--(3.3,0)--(3.3,2.5);
	\draw[thick] (2.5,2)--(2.5,1);
	\draw[dashed] (2.5,2)--(3.1,2)--(3.1,0.5);
	\end{tikzpicture}
	\caption{Reflections in the second part of the proof of 
Lemma~\ref{lem:zigzag}.}
	\label{fig:refl2}
\end{figure}
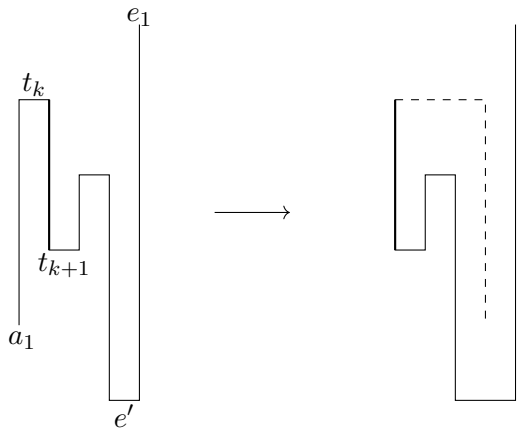	
	
	If $f(a)\geq f(e')$ for each $a\in[0, a_1]$ (or if $f(e)\leq f(a')$ for 
all $e\in[e_1, 1]$ in the second case), we are done. So assume that there is 
$a_2\in[0, a_1]$ such that $f(a_2)<f(e')$ and take the largest such $a_2$. Then 
there exists $a''\in[a_2, a_1]$ such that $f(a'')\geq f(e_1)$, and take $a''$ for which $f|_{[a_2, a_1]}$ attains its global maximum.  If $f(e)\leq f(a'')$ for each $e\in[e_1, 1]$, we 
reflect $f|_{[a_2, a'']}$ over $f|_{[e_1, 1]}$ and are done. If there is 
(minimal) $e_2>e_1$ such that $f(e_2)>f(a'')$, then there exists $e''\in[e_1, 
e_2]$ such that $f(e'')\leq f(a_2)$ and for which $f|_{[e_1, 
e_2]}$ attains a global minimum. In that case we reflect $f|_{[a'', t_k]}$ over $f|_{[t_k, e']}$ and 
$f|_{[a_2, a'']}$ over $f|_{[t_k,e'']}$ (see Figure~\ref{fig:refl3}). 

\begin{figure}[!ht]
	\centering
	\begin{tikzpicture}[scale=2]
	\draw (-0.4,-0.2)--(-0.4,2.7)--(-0.2,2.7)--(-0.2,0.5)-- 
(0,0.5)--(0,2)--(0.2,2)--(0.2,1)--(0.4,1)--(0.4,1.5)--(0.6,1.5)--(0.6,0)--(0.8,
0)--(0.8,2.5)--(1,2.5)--(1,-0.5)--(1.2,-0.5)--(1.2,3);
	\draw[thick] (0.2,2)--(0.2,1);
	\node at (-0.25,2.8) {\small $a''$};
	\node at (-0.4,-0.3) {\small $a_2$};
	\node at (-0.05,0.4) {\small $a_1$};
	\node at (0.1,2.1) {\small $t_k$};
	\node at (0.3,0.9) {\small $t_{k+1}$};
	\node at (0.7,-0.1) {\small $e'$};
	\node at (0.95,2.55) {\small $e_1$};
	\node at (1.1,-0.6) {\small $e''$};
	\node at (1.2,3.05) {\small $e_2$};
	
	\draw[->] (1.3,1.25)--(1.8,1.25);
	\draw(2.5,2)--(2.5,1)-- 
(2.7,1)--(2.7,1.5)--(2.9,1.5)--(2.9,0)--(3.5,0)--(3.5,2.5)--(3.7,2.5)--(3.7,
-0.5)--(4.1,-0.5)--(4.1,3);
	\draw[thick] (2.5,2)--(2.5,1);
	\draw[gray] (2.5,2)--(3.1,2)--(3.1,0.5)--(3.3,0.5)--(3.3,2.7);
	\draw[dashed] (3.3,2.7)--(3.9,2.7)--(3.9,-0.2);
	\end{tikzpicture}
	\caption{Reflections in the third part of the proof of 
Lemma~\ref{lem:zigzag}.}
	\label{fig:refl3}
\end{figure}
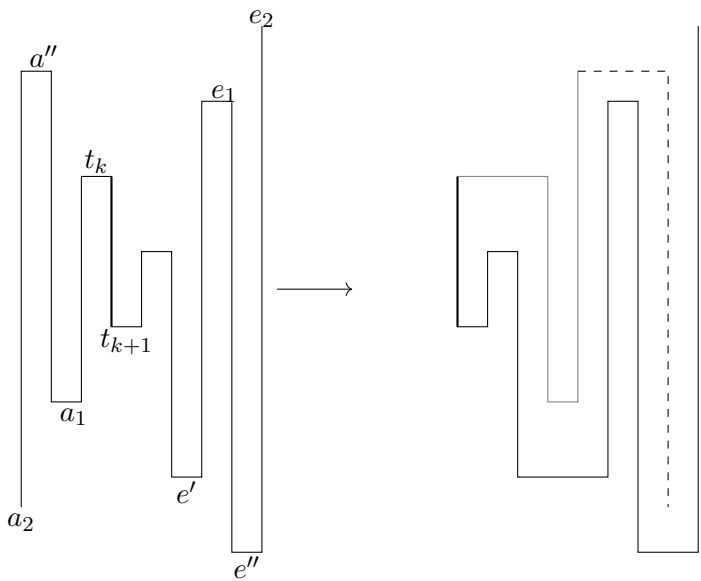

Thus we have constructed a permutation such that $H_k$ becomes the top of 
$G_{f|_{{[a_2, e_2]}}}$. We proceed by induction.
\end{proof}

\begin{theorem}\label{thm:zigzag}
	Let $X_{\infty}=\underleftarrow{\lim}\{I, f_i\}$ where each $f_i: I\to I$ 
is a continuous piecewise linear surjection. 
If $x=(x_0, x_1, x_2, \dots)\in X_{\infty}$ is such that for each $i\in\N$, $x_i$ is not inside a 
zigzag of $f_i$, then there exists an embedding of $X_{\infty}$ 
in the plane such that $x$ is accessible.
\end{theorem}
\begin{proof}
	The proof follows by Lemma~\ref{lem:zigzag} and 
Theorem~\ref{thm:algorithm}.
\end{proof}

\begin{corollary}\label{cor:nonzigzag}
	Let $X_{\infty}=\underleftarrow{\lim}\{I, f_i\}$ where each $f_i: I\to I$ 
is a continuous piecewise linear surjection which does not have zigzags. Then, for every $x\in 
X_{\infty}$ there exists an embedding of $X_{\infty}$ in the plane such that $x$ 
is accessible.
\end{corollary}

\begin{remark}
	Note that if $T:I\to I$ is a unimodal map and $x\in\UIL$, then $\UIL$ 
can be embedded in the plane such that $x$ is accessible by the previous 
corollary. This is Theorem 1 of \cite{embed}. This easily generalizes to an 
inverse limit of open interval maps (\eg generalized Knaster continua). 
\end{remark}

The following lemma shows that given arbitrary chains $\{C_i\}$, the zigzag 
condition from Lemma~\ref{lem:zigzag} cannot be improved.

\begin{lemma}
	Let $f: I\to I$ be a continuous piecewise linear surjection with critical points $0=t_0<t_1<\ldots <t_m<t_{m+1}=1$. If $I_k=[t_k, 
t_{k+1}]$ is inside a zigzag for some $k\in\{0, \ldots, m\}$, then there exists a 
nice chain $C$ covering $I$ such that $p(k)\neq m$ for every admissible permutation 
$p$ of $G_f$ with respect to $C$.
\end{lemma}
\begin{proof}
	Take a nice chain cover $C$ of $I$ such that $\mesh\, 
(C)<\min\{|f(t_i)-f(t_j)|: i,j\in\{0, \ldots, m+1\}, f(t_i)\neq f(t_j)\}$. Assume 
without loss of generality that $f(t_k)>f(t_{k+1})$ and let 
$t_i<t_k<t_{k+1}<t_{j}$ be such that minimum and maximum 
of $f|_{[t_i, t_j]}$ are attained at $t_i$ and $t_j$ respectively. Assume $t_i$ is the largest and $t_{j}$ is the smallest index 
with such properties. Let $p$ be some permutation. If $p(i)<p(j)<p(k)$, then by 
the choice of $C$, $p(H_j)$ intersects $p(V_{i'})$ for some $i'\in\{i, \ldots, 
k\}$. We proceed similarly if $p(j)<p(i)<p(k)$.
\end{proof}

\begin{remark}
	Let $X_{\infty}=\underleftarrow{\lim}\{I, f_i\}$ and $x=(x_0, x_1, x_2, \ldots)\in X_{\infty}$. If there exist piecewise linear continuous 
surjections $g_i: I\to I$ and a homeomorphism $h: X_{\infty}\to 
\underleftarrow{\lim}\{I, g_i\}$ such that every projection of $h(x)$ is not in 
a zigzag of $g_i$, then $X_{\infty}$ can be embedded in the plane such that $x$ 
is accessible. We have the following two corollaries. See also 
Examples~\ref{ex:2sin1x}-\ref{ex:Nadler}.
\end{remark}

\begin{corollary}
	Let $X_{\infty}=\underleftarrow{\lim}\{I, f_i\}$ where each $f_i: I\to I$ 
is a continuous piecewise linear surjection. 
If $x=(x_0, x_1, x_2, \dots)\in X_{\infty}$ is such that $x_i$ is inside a 
zigzag of $f_i$ for at most finitely many $i\in\N$, then there exists an 
embedding of $X_{\infty}$ in the plane such that $x$ is accessible.
\end{corollary}
\begin{proof}
	Since $\underleftarrow{\lim}\{I, f_i\}$ and $\underleftarrow{\lim}\{I, 
f_{i+n}\}$ are homeomorphic for every $n\in\N$, the proof follows using 
Theorem~\ref{thm:zigzag}.
\end{proof}

\begin{corollary}\label{cor:iterations}
	Let $f$ be a continuous piecewise linear surjection with finitely many 
critical points and $x=(x_0, x_1, x_2, \dots)\in 
X_{\infty}=\underleftarrow{\lim}\{I, f\}$. If there exists $k\in\N$ such that 
$x_i$ is not inside a zigzag of $f^k$ for all but finitely many $i$ , then there 
exists a planar embedding of $X_{\infty}$ such that $x$ is accessible.
\end{corollary}	
\begin{proof}
	Note that $\underleftarrow{\lim}\{I, f^k\}$ and $X_{\infty}$ are 
homeomorphic.
\end{proof}

We give applications of Corollary~\ref{cor:iterations} in the following 
examples.

\begin{example}\label{ex:2sin1x}
	Let $f$ be a piecewise linear map such that $f(0)=0$, $f(1)=1$ and with 
critical points $\frac{1}{4}, \frac{3}{4}$, where $f(\frac{1}{4})=\frac{3}{4}$ 
and $f(\frac{3}{4})=\frac{1}{4}$ (see Figure~\ref{fig:spiral}). 
	
	\begin{figure}[!ht]
		\centering
		\begin{tikzpicture}[scale=3]
		\draw[dashed] (0,0)--(1,1);
		\draw (0,0) -- (1/4, 3/4) -- (3/4, 1/4) -- (1,1);
		\draw (0,0) -- (0,1) -- (1,1) -- (1,0) -- (0,0);
		\draw[dashed] (1/4, 0) -- (1/4, 1);
		\draw[dashed] (1/2, 0) -- (1/2, 1);
		\draw[dashed] (3/4, 0) -- (3/4, 1);
		\draw[dashed] (0, 1/4) -- (1, 1/4);
		\draw[dashed] (0, 1/2) -- (1, 1/2);
		\draw[dashed] (0, 3/4) -- (1, 3/4);
		
		\node at (1/4,-0.1) {$\frac{1}{4}$};
		\node at (1/2,-0.1) {$\frac{1}{2}$};
		\node at (3/4,-0.1) {$\frac{3}{4}$};
		
		\node at (-0.1,1/4) {$\frac{1}{4}$};
		\node at (-0.1,1/2) {$\frac{1}{2}$};
		\node at (-0.1,3/4) {$\frac{3}{4}$};
		\draw[solid, fill] (1/2,0.5) circle (0.02);
		
		\end{tikzpicture}
		\caption{Graph of $f$ from Example~\ref{ex:2sin1x}.}
		\label{fig:spiral}
	\end{figure}
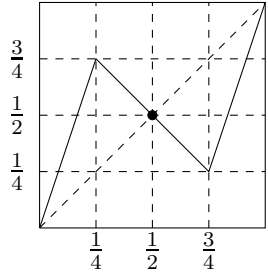
	
	Note that $X=\underleftarrow{\lim}\{I,f\}$ consists of two rays compactifying 
on an arc and therefore, for every $x\in X$, 
	there exists a planar embedding making $x$ accessible. However, the 
point $\frac{1}{2}$ is inside a zigzag of $f$. 
Figure~\ref{fig:squared} shows the graph of $f^2$. Note that the 
point $\frac{1}{2}$ is not inside a zigzag of $f^2$ and that gives an 
embedding of $X$ such that $(\frac{1}{2}, \frac{1}{2}, \ldots)$ is accessible. 
	
	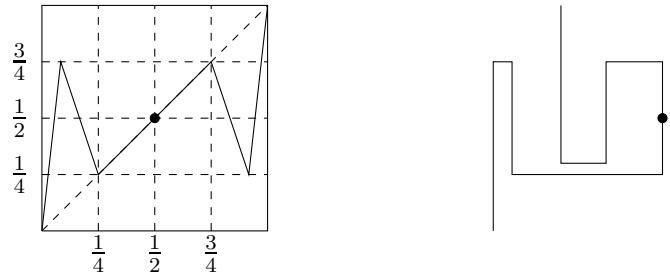
\begin{figure}[!ht]
		\centering
		\begin{tikzpicture}[scale=3]
		\draw[dashed] (0,0)--(1,1);
		\draw (0,0) -- (1/12, 3/4) -- (1/4,1/4) -- (3/4, 3/4) -- 
(11/12,1/4) -- (1,1);
		\draw (0,0) -- (0,1) -- (1,1) -- (1,0) -- (0,0);
		\draw[dashed] (1/4, 0) -- (1/4, 1);
		\draw[dashed] (1/2, 0) -- (1/2, 1);
		\draw[dashed] (3/4, 0) -- (3/4, 1);
		\draw[dashed] (0, 1/4) -- (1, 1/4);
		\draw[dashed] (0, 1/2) -- (1, 1/2);
		\draw[dashed] (0, 3/4) -- (1, 3/4);
		
		\node at (1/4,-0.1) {$\frac{1}{4}$};
		\node at (1/2,-0.1) {$\frac{1}{2}$};
		\node at (3/4,-0.1) {$\frac{3}{4}$};
		
		\node at (-0.1,1/4) {$\frac{1}{4}$};
		\node at (-0.1,1/2) {$\frac{1}{2}$};
		\node at (-0.1,3/4) {$\frac{3}{4}$};
		
		\draw[solid, fill] (1/2,0.5) circle (0.02);
		
		\draw 
(2,0)--(2,3/4)--(2+1/12,3/4)--(2+1/12,1/4)--(2+3/4,1/4)--(2+3/4,3/4)--(2+1/2,
3/4)--(2+1/2,0.3)--(2+0.3,0.3)--(2+0.3,1);
		\draw[solid, fill] (2.75,0.5) circle (0.02);
		\end{tikzpicture}
		\caption{Graph of $f^2$ from Example~\ref{ex:2sin1x}.}
		\label{fig:squared}
	\end{figure}
	Let  $x=(x_0, x_1, x_2, \ldots)\in X$ be such that $x_i\in[1/4,3/4]$ 
for all but finitely many $i\in\N_0$.
	Then, the embedding in Figure~\ref{fig:squared} will make $x$ 
accessible. 
	For other points $x=(x_0, x_1, x_2, \ldots)\in X$ there exists $N\in\N$ 
such that $x_i\in[0,1/4]$ 
	for each $i>N$ or $x_i\in[3/4,1]$ for each $i>N$ so the standard 
embedding makes them accessible. 
	In fact, the embedding in Figure~\ref{fig:squared} will make every 
$x\in X$ accessible. 
\end{example}

\begin{example}\label{ex:2knsater}
	Assume that $f$ is a piecewise linear map with $f(0)=0$, $f(1)=1$ and 
critical points $f(\frac{3}{8})=\frac{3}{4}$ and $f(\frac{5}{8})=\frac{1}{4}$ 
(see Figure~\ref{fig:doubleKnaster}).
	
	\begin{figure}[!ht]
		\centering
		\begin{tikzpicture}[scale=3]
		\draw[dashed] (0,0)--(1,1);
		\draw (0,0) -- (3/8, 3/4) -- (5/8, 1/4) -- (1,1);
		\draw (0,0) -- (0,1) -- (1,1) -- (1,0) -- (0,0);
		\draw[dashed] (1/4, 0) -- (1/4, 1);
		\draw[dashed] (1/2, 0) -- (1/2, 1);
		\draw[dashed] (3/4, 0) -- (3/4, 1);
		\draw[dashed] (0, 1/4) -- (1, 1/4);
		\draw[dashed] (0, 1/2) -- (1, 1/2);
		\draw[dashed] (0, 3/4) -- (1, 3/4);
		
		\node at (1/4,-0.1) {$\frac{1}{4}$};
		\node at (1/2,-0.1) {$\frac{1}{2}$};
		\node at (3/4,-0.1) {$\frac{3}{4}$};
		
		\node at (-0.1,1/4) {$\frac{1}{4}$};
		\node at (-0.1,1/2) {$\frac{1}{2}$};
		\node at (-0.1,3/4) {$\frac{3}{4}$};
		
		\draw[solid, fill] (0.5,0.5) circle (0.02);
		\end{tikzpicture}
		\begin{tikzpicture}[scale=3]
		\draw[dashed] (0,0)--(1,1);
		\draw (0,0) -- (3/16, 3/4) -- (1/4+1/16, 1/4) -- (3/8,1/2) -- 
(1/2-1/16,1/4) -- (1/2+1/16,3/4) -- (1/2+1/8,1/2) -- (3/4-1/16,3/4) -- 
(3/4,1/2) -- (3/4+1/16,1/4) -- (1,1);
		\draw (0,0) -- (0,1) -- (1,1) -- (1,0) -- (0,0);
		\draw[dashed] (1/4, 0) -- (1/4, 1);
		\draw[dashed] (1/2, 0) -- (1/2, 1);
		\draw[dashed] (3/4, 0) -- (3/4, 1);
		\draw[dashed] (0, 1/4) -- (1, 1/4);
		\draw[dashed] (0, 1/2) -- (1, 1/2);
		\draw[dashed] (0, 3/4) -- (1, 3/4);
		
		\node at (1/4,-0.1) {$\frac{1}{4}$};
		\node at (1/2,-0.1) {$\frac{1}{2}$};
		\node at (3/4,-0.1) {$\frac{3}{4}$};
		
		\node at (-0.1,1/4) {$\frac{1}{4}$};
		\node at (-0.1,1/2) {$\frac{1}{2}$};
		\node at (-0.1,3/4) {$\frac{3}{4}$};
		
		\draw[solid, fill] (0.5,0.5) circle (0.02);
		
		\end{tikzpicture}
		\caption{Graph of $f$ and $f^2$ in Example~\ref{ex:2knsater}.}
		\label{fig:doubleKnaster}
	\end{figure}
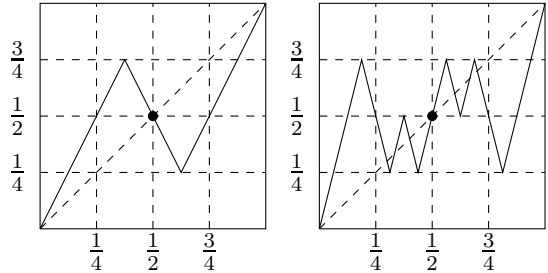
	
	Note that $X=\underleftarrow{\lim}\{I,f\}$ consists of two Knaster 
continua joined at their endpoints 
	together with two rays both converging to these two Knaster continua. 
Note that 
	$(\frac{1}{2}, \frac{1}{2}, \ldots)$ 
	can be embedded accessibly with the use of $f^2$, see 
Figure~\ref{fig:doubleKnaster}. 
	However, as opposed to the previous example, $X$ cannot be embedded 
such that every point is accessible (this follows already by a result of Mazurkiewicz \cite{Maz} which says that there exist nonaccessible points in any planar embedding of a nondegenerate indecomposable continuum).
 It is proven by Minc and Transue in 
\cite{MincTrans} that such an embedding of a chainable continuum exists if and 
only if it is {\em Suslinean}, \ie
	contains at most countably many mutually disjoint nondegenerate 
subcontinua.
\end{example}

\begin{example}[Nadler]\label{ex:Nadler}
Let $f: I\to I$ be as in Figure~\ref{fig:iter}. This is Nadler's 
candidate from \cite{Nadler} for a negative answer to Question~\ref{q:NaQu}. 
However, in what follows we show that every point can be made accessible via some planar 
embedding of $\underleftarrow{\lim}(I,f)$.
	
	Let $n\in\N$. If $J\subset I$ is a maximal interval such that $f^n|_J$ 
is increasing, then $J$ is not inside a zigzag of $f^n$, see \eg 
Figure~\ref{fig:iter}.
	
	\begin{figure}[!ht]
		\centering
		\begin{tikzpicture}[scale=3]
		\draw[dashed] (0,0)--(1,1);
		\draw (0,0) -- (1/5, 1/5) -- (2/5, 4/5) -- (3/5,1/5) -- (4/5, 
4/5) -- (1,1);
		\draw (0,0) -- (0,1) -- (1,1) -- (1,0) -- (0,0);
		\draw[dashed] (1/5, 0) -- (1/5, 1);
		\draw[dashed] (2/5, 0) -- (2/5, 1);
		\draw[dashed] (3/5, 0) -- (3/5, 1);
		\draw[dashed] (4/5, 0) -- (4/5, 1);
		\draw[dashed] (0, 1/5) -- (1, 1/5);
		\draw[dashed] (0, 2/5) -- (1, 2/5);
		\draw[dashed] (0, 3/5) -- (1, 3/5);
		\draw[dashed] (0, 4/5) -- (1, 4/5);
		
		\node at (1/5,-0.1) {$\frac{1}{5}$};
		\node at (2/5,-0.1) {$\frac{2}{5}$};
		\node at (3/5,-0.1) {$\frac{3}{5}$};
		\node at (4/5,-0.1) {$\frac{4}{5}$};
		
		\node at (-0.1,1/5) {$\frac{1}{5}$};
		\node at (-0.1,2/5) {$\frac{2}{5}$};
		\node at (-0.1,3/5) {$\frac{3}{5}$};
		\node at (-0.1,4/5) {$\frac{4}{5}$};
		
		\draw[thick] (1/5,1/5)--(2/5,4/5);
		\draw[thick] (3/5,1/5)--(4/5,4/5);
		
		\end{tikzpicture}
		\begin{tikzpicture}[scale=3]
		\draw[dashed] (0,0)--(1,1);
		\draw (0,0) -- (1/5, 1/5) -- (4/15, 4/5) -- (5/15,1/5) -- 
(6/15, 4/5) -- (7/15, 1/5) -- (8/15,4/5) -- (9/15, 1/5) -- (10/15, 4/5) -- 
(11/15,1/5) -- (12/15, 4/5) -- (1,1);
		\draw (0,0) -- (0,1) -- (1,1) -- (1,0) -- (0,0);
		\draw[dashed] (1/5, 0) -- (1/5, 1);
		\draw[dashed] (2/5, 0) -- (2/5, 1);
		\draw[dashed] (3/5, 0) -- (3/5, 1);
		\draw[dashed] (4/5, 0) -- (4/5, 1);
		\draw[dashed] (0, 1/5) -- (1, 1/5);
		\draw[dashed] (0, 2/5) -- (1, 2/5);
		\draw[dashed] (0, 3/5) -- (1, 3/5);
		\draw[dashed] (0, 4/5) -- (1, 4/5);
		
		\node at (1/5,-0.1) {$\frac{1}{5}$};
		\node at (2/5,-0.1) {$\frac{2}{5}$};
		\node at (3/5,-0.1) {$\frac{3}{5}$};
		\node at (4/5,-0.1) {$\frac{4}{5}$};
		
		\node at (-0.1,1/5) {$\frac{1}{5}$};
		\node at (-0.1,2/5) {$\frac{2}{5}$};
		\node at (-0.1,3/5) {$\frac{3}{5}$};
		\node at (-0.1,4/5) {$\frac{4}{5}$};
		
		\draw[thick] (1/5,1/5)--(4/15,4/5);
		\draw[thick] (5/15,1/5)--(6/15,4/5);
		\draw[thick] (7/15,1/5)--(8/15,4/5);
		\draw[thick] (9/15,1/5)--(10/15,4/5);
		\draw[thick] (11/15,1/5)--(12/15,4/5);
		\end{tikzpicture}
		\caption{Map $f$ and its second iterate. Bold lines are 
increasing branches of the restriction to $[ 1/5, 4/5]$. Note that they are not 
inside a zigzag of $f$ or $f^2$ respectively.}
		\label{fig:iter}
	\end{figure}
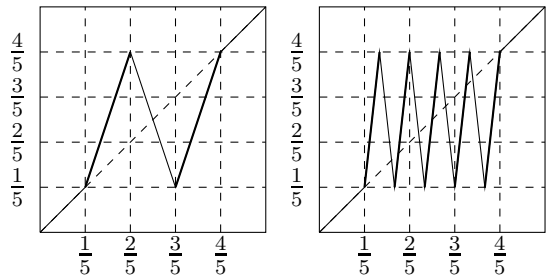

	We will code the orbit of points in the invariant interval $[ 1/5, 
4/5]$ in the following way. For $y\in [ 1/5, 4/5]$ let 
$i(y)=(y_n)_{n\in\N_0}\subset\{0, 1, 2\}^{\infty}$, where
	$$y_n=\begin{cases}
	0,& f^n(y)\in[ 1/5, 2/5],\\
	1,& f^n(y)\in[ 2/5, 3/5],\\
	2,& f^n(y)\in[ 3/5, 4/5].
	\end{cases}$$
	The definition is somewhat ambiguous with a problem occurring at points 
$2/5$ and $3/5$. Note, however, that $f^n(2/5)=4/5$ and $f^n(3/5)=1/5$ for all 
$n\in\N$. So every point in $[1/5, 4/5]$ will have a unique itinerary, except 
the preimages of $2/5$ (to which we can assign two itineraries $a_1\ldots 
a_n\frac 01 2222\ldots$) and preimages of $3/5$, (to which we can assign two 
itineraries $a_1\ldots a_n\frac 12 0000\ldots$), where $\frac 01$ means ``$0$ 
or $1$'' and $a_1, \ldots, a_n\in\{0,1,2\}$.
	
	Note that if $i(y)=1y_2\ldots y_n1$, where $y_i\in\{0, 2\}$ for every 
$i\in\{2, \ldots, n\}$, then $y$ is contained in an increasing branch of 
$f^{n+1}$. This holds also if $n=1$, \ie $y_2\ldots y_n=\emptyset$. Also, if 
$i(y)=0\ldots$ or $i(y)=2\ldots$, then $y$ is contained in an increasing branch 
of $f$. See Figure~\ref{fig:coding}.
	
	\begin{figure}[!ht]
		\centering
		\begin{tikzpicture}[scale=4]
		\draw[dashed] (0,0)--(1,1);
		\draw (0,0) -- (1/5, 1/5) -- (2/5, 4/5) -- (3/5,1/5) -- (4/5, 
4/5) -- (1,1);
		\draw (0,0) -- (0,1) -- (1,1) -- (1,0) -- (0,0);
		\draw[dashed] (1/5, 0) -- (1/5, 1);
		\draw[dashed] (2/5, 0) -- (2/5, 1);
		\draw[dashed] (3/5, 0) -- (3/5, 1);
		\draw[dashed] (4/5, 0) -- (4/5, 1);
		\draw[dashed] (0, 1/5) -- (1, 1/5);
		\draw[dashed] (0, 2/5) -- (1, 2/5);
		\draw[dashed] (0, 3/5) -- (1, 3/5);
		\draw[dashed] (0, 4/5) -- (1, 4/5);
		
		\node at (3/10,-0.1) {\small $0$};
		\node at (5/10,-0.1) {\small $1$};
		\node at (7/10,-0.1) {\small $2$};
		
		\draw[thick] (1/5,1/5)--(2/5,4/5);
		\draw[thick] (3/5,1/5)--(4/5,4/5);
		
		\end{tikzpicture}
		\begin{tikzpicture}[scale=4]
		\draw[dashed] (0,0)--(1,1);
		\draw (0,0) -- (1/5, 1/5) -- (4/15, 4/5) -- (5/15,1/5) -- 
(6/15, 4/5) -- (7/15, 1/5) -- (8/15,4/5) -- (9/15, 1/5) -- (10/15, 4/5) -- 
(11/15,1/5) -- (12/15, 4/5) -- (1,1);
		\draw (0,0) -- (0,1) -- (1,1) -- (1,0) -- (0,0);
		\draw[dashed] (1/5, 0) -- (1/5, 1);
		\draw[dashed] (2/5, 0) -- (2/5, 1);
		\draw[dashed] (3/5, 0) -- (3/5, 1);
		\draw[dashed] (4/5, 0) -- (4/5, 1);
		\draw[dashed] (0, 1/5) -- (1, 1/5);
		\draw[dashed] (0, 2/5) -- (1, 2/5);
		\draw[dashed] (0, 3/5) -- (1, 3/5);
		\draw[dashed] (0, 4/5) -- (1, 4/5);
		
		\node at (7/30,-0.1) {\rotatebox[]{90}{\tiny $00$}};
		\node at (9/30,-0.1) {\rotatebox[]{90}{\tiny $01$}};
		\node at (11/30,-0.1) {\rotatebox[]{90}{\tiny $02$}};
		\node at (13/30,-0.1) {\rotatebox[]{90}{\tiny $12$}};
		\node at (15/30,-0.1) {\rotatebox[]{90}{\tiny $11$}};
		\node at (17/30,-0.1) {\rotatebox[]{90}{\tiny $10$}};
		\node at (19/30,-0.1) {\rotatebox[]{90}{\tiny $20$}};
		\node at (21/30,-0.1) {\rotatebox[]{90}{\tiny $21$}};
		\node at (23/30,-0.1) {\rotatebox[]{90}{\tiny $22$}};
		
		\draw[thick] (1/5,1/5)--(4/15,4/5);
		\draw[thick] (5/15,1/5)--(6/15,4/5);
		\draw[thick] (7/15,1/5)--(8/15,4/5);
		\draw[thick] (9/15,1/5)--(10/15,4/5);
		\draw[thick] (11/15,1/5)--(12/15,4/5);
		\end{tikzpicture}
		\caption{Map $f$ and its iterate with symbolic coding of 
points. Note that points with itinerary $0\ldots$ or $2\ldots$ are contained in 
an increasing branch of $f$ and points with itineraries $11\ldots$ are contained 
in an increasing branch of $f^2$.}
		\label{fig:coding}
	\end{figure}
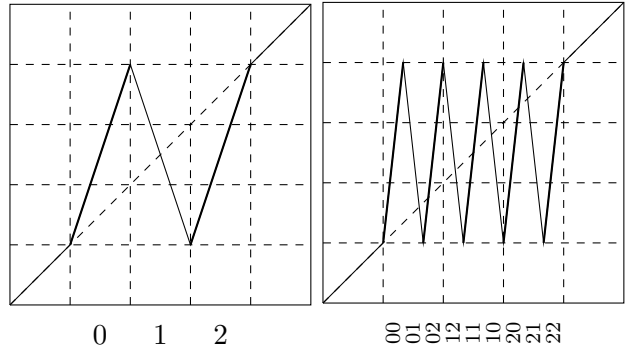
We extend the symbolic coding to $X$. For $x=(x_0, x_1, x_2, \ldots)\in X$ with itinerary $i(x)$
let $(y_k)_{k\in\Z}$ be defined by
$$
y_k=\begin{cases}
        i(x)_k,& k \geq 0, \text{ and for } k < 0,\\
	0,& x_{-k}\in[ 1/5, 2/5],\\
	1,& x_{-k}\in[ 2/5, 3/5],\\
	2,& x_{-k}\in[ 3/5, 4/5].
	\end{cases}
$$ 
Again, the assignment is injective everywhere 
except at preimages of critical points 
	$2/5$ or $3/5$.
	
Now fix $x=(x_0, x_1, x_2, \ldots)\in X$ with its backward itinerary 
$\ovl x=\ldots y_{-2}y_{-1}y_0$
	(assume the itinerary is unique, otherwise choose one of the two 
possible backward itineraries). 
	Assume first that $y_k\in\{0, 2\}$ for every $k\leq 0$. Then, for every 
$k\in\N_0$ it holds 
	that $i(x_k)=0\ldots$ or $i(x_k)=2\ldots$ so $x_k$ is in an increasing 
branch of $f$ and thus not inside a zigzag of $f$. By Theorem~\ref{thm:zigzag} it follows 
that there is an embedding 
	making $x$ accessible. Similarly, if there exists $n\in\N$ such that 
$y_k\neq 1$ for $k<-n$, we use that 
	$X$ is homeomorphic to $\underleftarrow{\lim}\{I, f_j\}$ where 
$f_1=f^n$, $f_j=f$ for $j\geq 2$.
	
	Assume that $\ovl x=\ldots 
1(\frac02)^{n_3}1(\frac02)^{n_2}1(\frac02)^{n_1}$ where $\frac02$ 
	means ``$0$ or $2$'' and $n_i\geq 0$ for $i\in\N$. We will assume that 
$n_1>0$; 
	the general case follows similarly. Note that 
$i(x_{n_1-1})=(\frac02)^{n_1}\ldots$ and so it is 
	contained in an increasing branch of $f^{n_1-1}$. 
	Note further that 
$i(x_{n_1+1+n_2})=1(\frac02)^{n_2}1(\frac02)^{n_1}\ldots$ and so it is 
contained 
	in an increasing branch of $f^{n_2+2}$. Also 
$f^{n_2+2}(x_{n_1+1+n_2})=x_{n_1-1}$. 
	Further we note that 
$i(x_{n_1+1+n_2+1+n_3-1})=(\frac02)^{n_3}1(\frac02)^{n_2}1(\frac02)^{n_1}$ and 
	so it is contained in an increasing branch of $f^{n_3}$. 
	Furthermore,  $f^{n_3}(x_{n_1+1+\\n_2+1+n_3-1})=x_{n_1+1+n_2}$.
	
In this way, we see that for every even $k\geq 4$ it holds that 
	
$$i(x_{n_1+1+n_2+1+\ldots+1+n_k})=1\left(\frac02\right)^{n_k}
1\ldots1\left(\frac02\right)^{n_1}\ldots$$
	and so it is contained in an increasing branch of $f^{n_k+2}$. Also, 
$f^{n_k+2}(x_{n_1+1+n_2+1+\ldots+1+n_k})=x_{n_1+1+n_2+1+\ldots+1+n_{k-1}-1}$. 
Similarly, 
$$i(x_{n_1+1+n_2+1+\ldots+1+n_k+1+n_{k+1}-1})=\left(\frac02\right)^{n_{k+1}}
1\ldots1\left(\frac02\right)^{n_1}\ldots$$ so 
$x_{n_1+1+n_2+1+\ldots+1+n_k+1+n_{k+1}-1}$ is in an increasing branch of 
$f^{n_{k+1}}$. Note also that 
$f^{n_{k+1}}(x_{n_1+1+n_2+1+\ldots+1+n_k+1+n_{k+1}-1})=x_{
n_1+1+n_2+1+\ldots+1+n_k}$.
	
	So we have the following sequence
$$
	\ldots \overset{f^{n_5}}{\longrightarrow} x_{n_1+1+\ldots 
+1+n_4}\overset{f^{n_4+2}}{\longrightarrow} 
x_{n_1+1+n_2+1+n_3-1}\overset{f^{n_3}}{\longrightarrow} 
x_{n_1+1+n_2}\overset{f^{n_2+2}}{\longrightarrow}x_{n_1-1}\overset{f^{n_1-1}}{\longrightarrow}x_0,
$$
	where the chosen points in the sequence are not contained in zigzags of 
the corresponding bonding maps. 
	Let 
	$$f_i=\begin{cases}
	f^{n_1-1},& i=1,\\
	f^{n_i+2},& i \text{ even,}\\
	f^{n_i},& i>1 \text{ odd}.
	\end{cases}$$
	Then, $\underleftarrow{\lim}\{I, f_i\}$ is homeomorphic to $X$ and by 
Theorem~\ref{thm:zigzag} it 
	can be embedded in the plane such that every 
$x\in\underleftarrow{\lim}\{I, f_i\}$ is accessible.
\end{example}

\section{Thin embeddings}\label{sec:thin}

We have proven that if a chainable continuum $X$ has an inverse limit 
representation such that $x\in X$ 
is not contained in zigzags of bonding maps, then there is a planar embedding 
of $X$ making $x$ accessible. 
Note that the converse is not true. The pseudo-arc is a counter-example, because it is homogeneous, so 
each of its points can be embedded accessibly. 
However, the crookedness of the bonding maps producing the 
pseudo-arc implies the occurrence of 
zigzags in every representation. 
Since the pseudo-arc is hereditarily indecomposable, no point is contained in an arc. 
To the contrary, in Minc's continuum $X_M$ (see Figure~\ref{fig:Minc}), 
every point is contained in an arc of length at least $\frac{1}{3}$. 

\begin{figure}[!ht]
	\centering
	\begin{tikzpicture}[scale=3]
	
	\draw (0,0) -- (1/3, 1) -- (5/12, 1/3) -- (7/12,2/3) -- (2/3, 0) -- 
(1,1);
	\draw (0,0) -- (0,1) -- (1,1) -- (1,0) -- (0,0);
	\draw[dashed] (1/3, 0) -- (1/3, 1);
	\draw[dashed] (1/2, 0) -- (1/2, 1);
	\draw[dashed] (2/3, 0) -- (2/3, 1);
	\draw[dashed] (0, 1/3) -- (1, 1/3);
	\draw[dashed] (0, 1/2) -- (1, 1/2);
	\draw[dashed] (0, 2/3) -- (1, 2/3);
	
	\node at (1/3,-0.1) {\small $\frac{1}{3}$};
	\node at (1/2,-0.1) {\small $\frac{1}{2}$};
	\node at (2/3,-0.1) {\small $\frac{2}{3}$};
	
	\node at (-0.1,1/3) {\small $\frac{1}{3}$};
	\node at (-0.1,1/2) {\small $\frac{1}{2}$};
	\node at (-0.1,2/3) {\small $\frac{2}{3}$};
	
	\draw[solid, fill] (1/2,0.5) circle (0.015);
	\node at (0.9,0.9) {\small $f$};
	\node at (0.53,0.45) {\scriptsize $p$};
	\end{tikzpicture}
	\begin{tikzpicture}[scale=3]
	
	\draw (0,0) -- (1/9, 1) -- (5/36, 1/3) -- (7/36,2/3) -- (2/9, 0) -- 
(1/3,1)--(3/8,0)--(3/8+1/96,2/3)--(10/24-1/96,1/3)--(10/24,1)--(5/12+1/24,
1/3)--(1/2,1/2);
	\begin{scope}[yscale=-1,xscale=-1,xshift=-28.5, yshift=-28.5]
	\draw (0,0) -- (1/9, 1) -- (5/36, 1/3) -- (7/36,2/3) -- (2/9, 0) -- 
(1/3,1)--(3/8,0)--(3/8+1/96,2/3)--(10/24-1/96,1/3)--(10/24,1)--(5/12+1/24,
1/3)--(1/2,1/2);
	\end{scope}
	\draw (0,0) -- (0,1) -- (1,1) -- (1,0) -- (0,0);
	\draw[dashed] (1/3, 0) -- (1/3, 1);
	\draw[dashed] (1/2, 0) -- (1/2, 1);
	\draw[dashed] (2/3, 0) -- (2/3, 1);
	\draw[dashed] (0, 1/3) -- (1, 1/3);
	\draw[dashed] (0, 1/2) -- (1, 1/2);
	\draw[dashed] (0, 2/3) -- (1, 2/3);
	
	\node at (1/3,-0.1) {\small $\frac{1}{3}$};
	\node at (1/2,-0.1) {\small $\frac{1}{2}$};
	\node at (2/3,-0.1) {\small $\frac{2}{3}$};
	
	\node at (-0.1,1/3) {\small $\frac{1}{3}$};
	\node at (-0.1,1/2) {\small $\frac{1}{2}$};
	\node at (-0.1,2/3) {\small $\frac{2}{3}$};
	
	\draw[solid, fill] (1/2,0.5) circle (0.015);
	\node at (0.53,0.45) {\scriptsize $p$};
	\node at (0.9,0.9) {\small $f^2$};
	
	\end{tikzpicture}
	\vspace{-10pt}
	\caption{Minc's map and its second iteration 
	(the example was given at the Spring Topology and Dynamical Systems Conference 2001 
	in a talk by Minc entitled ''On embeddings of chainable continua into the plane'').}
	\label{fig:Minc}
\end{figure}
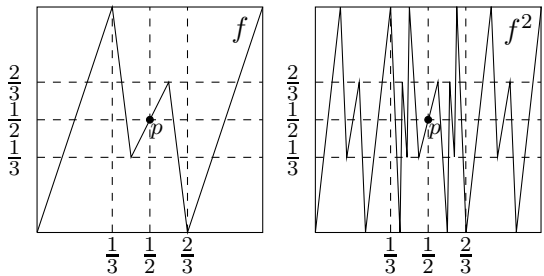

In the next two definitions we introduce the notion of thin embedding, used 
under this name in \eg \cite{DT}. In \cite{ABC-q} the notion of thin embedding 
was referred to as {\em $C$-embedding}. 

\begin{definition}
	Let $Y\subset\R^2$ be a continuum.  We say that $Y$ is {\em thin chainable} if there exists a 
	sequence $({\chain_n})_{n\in\N}$ of chains in $\R^2$ such that 
$Y=\cap_{n\in\N}\chain^*_n$, 
	where $\chain_{n+1}\prec\chain_n$ for every $n\in\N$, 
$\textrm{mesh}(\chain_n)\to 0$ as 
	$n\to\infty$, and the links of $\chain_n$ are connected sets in $\R^2$ 
(note that links are open in the 
	topology of $\R^2$). 
\end{definition}

\begin{definition}\label{def:thin}
	Let $X$ be a chainable continuum. We say that an embedding $\phi: 
X\to\R^2$ is a {\em thin embedding} 
	if $\phi(X)$ is thin chainable. Otherwise $\phi$ is called a {\em thick 
embedding}.
\end{definition}

Note that in \cite{Bing} Bing shows that every chainable continuum has a thin embedding in the plane.

\begin{question}[Minc, 2001]
	Is there a planar embedding of Minc's chainable continuum $X_M$ which makes $p$ accessible? Or as a special case, is there a 
\emph{thin embedding} of $X_M$ which makes $p$ accessible? 
\end{question}
%What are possible thin embeddings of $X_M$?

\begin{example}[Bing, \cite{Bing}]  An {\em Elsa continuum} (see \cite{Na-Elsa}) is a continuum consisting of a ray 
compactifying on an arc (in \cite{BrBr} this was called an arc+ray continuum). 
An example of a thick embedding of an Elsa continuum was constructed by Bing (see Figure~\ref{fig:Bing}). 
\end{example}

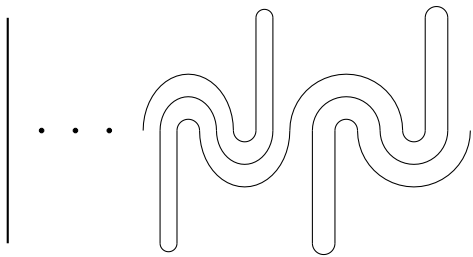
\begin{figure}[!ht]
	\centering
	\begin{tikzpicture}[scale=1.5]
	\draw[thick] (0,-1)--(0,1);
	
	\draw[domain=0:180] plot ({3+0.5*cos(\x)}, {0.5*sin(\x)});
	\draw[domain=0:180] plot ({3+0.3*cos(\x)}, {0.3*sin(\x)});
	\draw[domain=0:180] plot ({3+0.1*cos(\x)}, {0.1*sin(\x)});
	
	\draw[domain=180:360] plot ({2.1+0.4*cos(\x)}, {0.5*sin(\x)});
	\draw[domain=180:360] plot ({2.1+0.25*cos(\x)}, {0.3*sin(\x)});
	\draw[domain=180:360] plot ({2.1+0.1*cos(\x)}, {0.1*sin(\x)});
	\draw[domain=0:180] plot ({1.6+0.1*cos(\x)}, {0.1*sin(\x)});
	\draw[domain=0:180] plot ({1.6+0.25*cos(\x)}, {0.3*sin(\x)});
	\draw[domain=0:180] plot ({1.6+0.4*cos(\x)}, {0.5*sin(\x)});
	\draw (2.2,0)--(2.2,1);
	\draw (2.35,0)--(2.35,1);
	\draw[domain=0:180] plot ({2.275+0.075*cos(\x)}, {1+0.075*sin(\x)});
	\draw (1.5,0)--(1.5,-1);
	\draw (1.35,0)--(1.35,-1);
	\draw[domain=180:360] plot ({1.425+0.075*cos(\x)}, {-1+0.075*sin(\x)});
	
	\draw[domain=180:360] plot ({3.6+0.1*cos(\x)}, {0.1*sin(\x)});
	\draw[domain=180:360] plot ({3.6+0.3*cos(\x)}, {0.3*sin(\x)});
	\draw[domain=180:360] plot ({3.6+0.5*cos(\x)}, {0.5*sin(\x)});
	\draw (3.7,0)--(3.7,1);
	\draw (3.9,0)--(3.9,1);
	\draw[domain=0:180] plot ({3.8+0.1*cos(\x)}, {1+0.1*sin(\x)});
	\draw (2.9,0)--(2.9,-1);
	\draw (2.7,0)--(2.7,-1);
	\draw[domain=180:360] plot ({2.8+0.1*cos(\x)}, {-1+0.1*sin(\x)});
	\draw[solid, fill] (0.3,0) circle (0.02);
	\draw[solid, fill] (0.6,0) circle (0.02);
	\draw[solid, fill] (0.9,0) circle (0.02);
	\end{tikzpicture}
	\caption{Bing's example from \cite{Bing}.}
	\label{fig:Bing}
\end{figure}

An example of a thick embedding of the $3$-Knaster continuum was given by 
D\k{e}bski and Tymchatyn in \cite{DT}. 
An arc has a unique planar embedding (up to equivalence), so all of its planar embeddings are thin.
Therefore, it is natural to ask the following question.

\noindent
\begin{question} [Question~1 in \cite{ABC-q}] Which chainable continua have a thick 
embedding in the plane?
\end{question}

\begin{definition}
Given a chainable continuum $X$, let $\E_C(X)$ denote the 
set of all planar embeddings of $X$ obtained by performing admissible 
permutations of $G_{f_i}$ for every representation $X$ as 
$\underleftarrow{\lim}\{I, f_i\}$.
\end{definition}

The next theorem says that the class of all planar embeddings of chainable continuum $X$ obtained by 
performing admissible permutations of graphs $G_{f_i}$ 
is the class of all thin planar embeddings of $X$ up to the equivalence relation between embeddings. 

\begin{theorem}\label{thm:allemb}
	Let $X$ be a chainable continuum and $\phi: X\to\R^2$ a thin embedding 
of $X$. 
	Then there exists an embedding $\psi\in\E_C(X)$ which is equivalent to $\phi$.
\end{theorem}
\begin{proof}
Recall that $\chain^*_n = \bigcup_{\ell \in \chain_n} \ell$.
Let $\phi(X)=\cap_{n\in\N_0}\chain^*_n$, where the links of 
$\chain_n$ are open, connected sets in $\R^2$. Furthermore,  $\chain_{n+1}\prec\chain_n$ for every $n\in\N_0$. Note that we assume that links of $\chain_n$ have a polygonal curve for a boundary, using a brick decomposition of the plane. 

We argue that we can also assume that every $\chain^*_n$ is simply connected.
This goes in a few steps. In every step we first state the claim we can obtain and then argue in the rest of the step how to obtain it. 
\begin{enumerate}
 \item Without loss of generality we can assume that the separate
links of $\chain_n$ are simply connected, by filling in the holes. 
That is, if a link $\ell\in\chain_n$ 
is such that $\R^2\setminus\ell$ separates the plane, instead of $\ell$, we take $\ell \cup \bigcup_{i}V_i$, 
where the $V_i$ are the bounded components of $\R^2\setminus\ell$. 
Filling in the holes thus merges all the links contained in $\ell \cup \bigcup_i V_i$ into a single link.
This does not change the mesh nor the pattern of the chain. 

\iffalse
\item
We can assume without loss of generality that $\ell_{i+1}$ doesn't separate $\overline{\ell_i}$,
i.e., $\overline{\ell_i} \setminus \ell_{i+1}$ is connected.
Indeed, let $K_j$ be the components of $\overline{\ell_i} \cap \partial \ell_{i+1}$ that separate $\ell_i$, 
and take pairwise disjoint neighbourhoods $U_j$ of $K_j$ so small that $\ell_{i+1}$ doesn't separate
$\overline{U_j}$ and such that $\overline{U_j}\cap\overline{\ell_k}=\emptyset$ for all $k\neq i, i+1$ (recall that the chain is taut).
Pick a component $L$ of $\ell_i \setminus \ell_{i+1}$ that intersects $\ell_{i-1}$
(or just any component if $i = 1$).
Replace $\ell_i$ with $L \cup  U_j$, where  $K_j \subset \partial L$ and replace $\ell_{i+1}$ with $\ell_{i+1}\cup(\ell_i\setminus\overline{L})$.
The new cover is a chain, and $\ell_{i+1}$ no longer separates $\overline{\ell_i}$. 
\fi

\item We can assume that for every hole $H$ between the links $\ell_i$ and $\ell_{i+1}$ 
(\ie a connected bounded component of $\R^2\setminus (\ell_i\cup\ell_{i+1})$) 
it holds that if $H \cap \ell_j \neq \emptyset$ for some $j$, then 
$\ell_j \subset \ell_i \cup \ell_{i+1} \cup H$. Denote by $U_i$ the union of bounded component of 
$\R^2\setminus (\ell_i\cup\ell_{i+1})$ and note that 
$\{\ell_1\cup U_1\cup\ell_2, \ell_3\cup U_3\cup\ell_4, \ldots, \ell_{k(n)-1}\cup U_{k(n)-1}\cup\ell_{k(n)}\}$ 
is again a chain. (It can happen that the first or last few links are merged into one link. 
Also, we can assume that $n$ is even by merging the last two links if necessary.)
Denote for simplicity $\tilde{\ell}_i=\ell_{2i-1}\cup U_{2i-1}\cup\ell_{2i}$ for every $i\in \{1,2,\ldots, n/2\}$. 
We claim that if $\tilde{\ell}_j\cap H\neq\emptyset$ for some hole between $\tilde{\ell}_i$ 
and $\tilde{\ell}_{i+1}$, then $\tilde{\ell}_j\subset \tilde{\ell}_i\cup\tilde{\ell}_{i+1}\cup H$. 
Assume the contrary, and take without loss of generality that $j>i+1$. 
Then necessarily $j=i+2$ and $\tilde{\ell}_{i+2}$ separates $\tilde{\ell}_{i+1}$ so that at least 
two components of $\tilde \ell_{i+1} \setminus \tilde \ell_{i+2}$ intersect $\tilde{\ell}_i$. That is, $\tilde{\ell}_{i+2}$ separates $\ell_{2i+1}$. 
But this is a contradiction since $\tilde{\ell}_{i+2}=\ell_{2i+3}\cup U_{2i+3}\cup\ell_{2i+4}$ 
can only intersect $\ell_{2i+1}$ if $\ell_{2i+1}\subset U_{2i+3}$ in which case $\tilde{\ell}_{i+2}$ 
does not separate $\ell_{2i+1}$.

\item If there is a hole between links $\tilde \ell_i$ and $\tilde \ell_{i+1}$, then 
we can fill it in a similar way as in Step (1). That is, letting $\tilde U_i$ be the union 
of bounded components of
$\R^2 \setminus (\tilde \ell_i \cup \tilde \ell_{i+1})$, the links of the modified chain are 
$\{\tilde \ell_1, \ldots, \tilde \ell_{i - 1}, \tilde \ell_i\cup \tilde U_i\cup\tilde \ell_{i+1}, 
\tilde \ell_{i+2}, 
\ldots, \tilde \ell_{k(n)/2}\}$. (It can happen that $\tilde \ell_j\subset \tilde U_i$ for all $j>N\geq i+1$ or 
$j<N\leq i$, but then we merge all these links.)
We do this for each $i \in \{ 1, \dots, k(n)/2 \}$ where there is a hole between $\tilde \ell_i$ 
and $\tilde \ell_{i+1}$, so not just the odd values of $i$ as in Step (2).
Due to the claim in Step (2), the result is again a chain.
These modified chains can have a larger mesh (up to four times the original mesh), but still 
satisfy $\chain_{n+1}\prec\chain_n$ for every $n\in\N_0$ and $\mesh(\chain_n)\to 0$ as $n\to\infty$.
\end{enumerate}

In the rest of the proof we construct homeomorphisms $F_j$, $0 \leq j \leq n\in\N_0$
and $G_n:=F_n\circ\ldots \circ F_1\circ F_0$, which straighten the chains $\chain_n$ 
to horizontal chains. The existence of such homeomorphisms follows 
from the generalization of the piecewise linear Schoenflies' theorem given in 
\eg \cite[Section~3]{Moise}. 
Take a homeomorphism $F_0:\R^2\to\R^2$ which maps $\chain_0$ to a 
horizontal chain. 
Then $F_0(\chain_1)\prec F_0(\chain_0)$ 
and there is a homeomorphism $F_1:\R^2\to\R^2$ which is the identity on 
$\R^2\setminus F_0(\chain_0)^*$ 
(recall that $\chain_n^*$ denotes the union of links of $\chain_n$), and 
which maps $F_0(\chain_1)^*$ to a tubular neighborhood of some permuted 
flattened graph with $\mesh(F_1(F_0(\chain_1)))<\mesh(\chain_1)$.\\
Note that $G_n(\chain_{n+1})\prec G_n(\chain_{n})$ and there is a homeomorphism 
$F_{n+1}:\R^2\to\R^2$ which is the identity on $\R^2\setminus G_n(\chain_{n})^*$ 
and which maps $G_n(\chain_{n+1})^*$ to a tubular neighborhood of some 
flattened permuted graph with 
$\mesh(F_{n+1}(G_n(\chain_{n+1})))<\mesh(\chain_{n+1})$.\\
Note that the sequence $(G_n)_{n\in\N_0}$ is uniformly Cauchy and 
$G := \lim_{n\to\infty}G_n$ is well-defined. By construction, $G:\R^2\to\R^2$ is a 
homeomorphism and $G\circ\phi\in\E_C(X)$.
\end{proof}

\noindent
\begin{question}[Question~2 in \cite{ABC-q}] Is there a chainable continuum $X$ 
and a thick embedding $\psi$ of $X$ such that the set of accessible points of 
$\psi(X)$ is different from the set of accessible points of $\phi(X)$ for any 
thin embedding $\phi$ of $X$?
\end{question}

\section{Uncountably many nonequivalent embeddings}\label{sec:nonequivalent}

In this section we construct, for every chainable continuum containing a nondegenerate 
indecomposable subcontinuum, uncountably many embeddings which 
are pairwise not strongly equivalent. Recall that $\phi, 
\psi\colon X\to \R^2$ are strongly equivalent if $\phi\circ\psi^{-1}$ can be extended to 
a homeomorphism of $\R^2$.

The idea of the construction is to find uncountably many composants in some indecomposable planar continuum which can 
be embedded accessibly in more than a point. The conclusion then follows easily 
with the use of the following theorem.

\begin{theorem}[Mazurkiewicz \cite{Maz}]\label{thm:Maz}
Let $X\subset\R^2$ be an indecomposable  planar continuum. There are at 
most countably many composants of $X$ which are accessible in at least two 
points.
\end{theorem}

Let $X=\underleftarrow{\lim}\{I,f_i\}$, where $f_i: I\to I$ are continuous 
piecewise linear surjections. 

\begin{definition}
	Let $f: I\to I$ be a continuous surjection. An interval $I'\subset I$ is called a 
\emph{surjective interval} if $f(I')=I$ and $f(J)\neq I$ for every $J\subsetneq
I'$.
Let $A_1, \ldots, A_n$, $n \geq 1$, be the surjective intervals of $f$ ordered from left to right. 
	For every $i\in\{1, \ldots, n\}$ define the \emph{right accessible set} 
by $R(A_i) = \{ x\in A_i :  f(y)\neq f(x) \text{ for all } x<y\in A_i\}$ (see Figure~\ref{fig:LR}).
\end{definition}

\begin{figure}[!ht]
	\centering
	\begin{tikzpicture}[scale=4]
	\draw (0,0)--(1,0)--(1,1)--(0,1)--(0,0);
	\draw 
(0,0)--(1/9,0.6)--(2/9,0.4)--(1/3,1)--(0.4,0.8)--(0.45,1)--(2/3,0)--(7/9,
0.6)--(8/9,0.4)--(1,1);

	\begin{scope}
	[shift={(1,1)},rotate=180]
	\draw[->] (0,0.6)--(1/9,0.6);
	\draw[->] (0,5.4*0.02)--(0.02,5.4*0.02);
	\draw[->] (0,5.4*0.04)--(0.04,5.4*0.04);
	\draw[->] (0,5.4*0.06)--(0.06,5.4*0.06);
	\draw[->] (0,5.4*0.08)--(0.08,5.4*0.08);
	\draw[->] (0,5.4*0.1)--(0.1,5.4*0.1);
	\draw[->] (0,0.676)--(0.27,0.676);
	\draw[->] (0,0.78)--(0.29,0.78);
	\draw[->] (0,0.9)--(0.315,0.9);
	\draw[->] (0,0.97)--(0.325,0.97);
	\end{scope}
	
	\node at (1.1,1) {\small $f$};
	\draw[<->] (0,-0.1)--(1/3,-0.1);
	\node at (1/6,-0.2)  {\small $A_1$};
	\draw[<->] (0.45,-0.1)--(2/3,-0.1);
	\node at (0.55,-0.2)  {\small $A_2$};
	\draw[<->] (2/3,-0.1)--(1,-0.1);
	\node at (5/6,-0.2)  {\small $A_3$};
	
	\draw[thick] (0,0)--(1/9-0.035,0);
	\draw[dashed] (2/9,0.4)--(2/9,0);
	\draw[dashed] (1/9-0.035,0.4)--(1/9-0.035,0);
	\draw[dashed] (1/3,1)--(1/3,0);
	\draw[thick] (2/9,0)--(1/3,0);
	\node at (-0.2,0.24) {\tiny $R(A_1)$};
	\draw[->] (-0.2,0.2)--(0.05,0);
	\draw[->] (-0.2,0.2)--(0.3,0);

	\draw[thick] (2/3,0)--(0.74,0);
	\draw[dashed] (0.74,0.4)--(0.74,0);
	\draw[dashed] (8/9,0.4)--(8/9,0);
	\draw[thick] (8/9,0)--(1,0);
	\node at (1.2,0.24) {\tiny $R(A_3)$};
	\draw[->] (1.2,0.2)--(0.7,0);
	\draw[->] (1.2,0.2)--(8/9+1/18,0);
	
	\end{tikzpicture}
	\caption{Map $f$ has three surjective intervals. The right accessible 
sets in the surjective intervals $A_1$ and $A_3$ of $f$ are denoted in the 
picture by $R(A_1)$ and $R(A_3)$ respectively. Note that $R(A_2)=A_2$.}
	\label{fig:LR}
\end{figure}
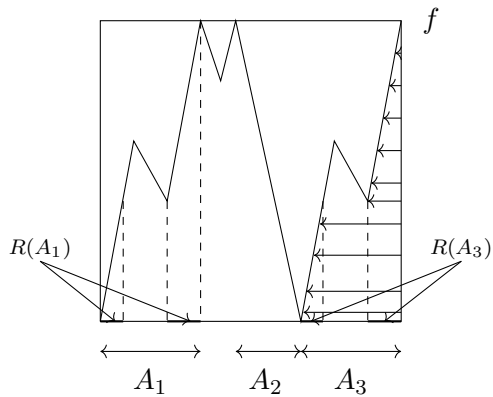

We will first assume that the map $f_i$ contains at least three surjective 
intervals for every $i\in\N$.
We will later see that this assumption can be made without loss of generality.

\begin{remark}\label{rem:image}
	Assume that $f:I\to I$ has $n\geq 3$ surjective intervals. 
	Then $A_1\cap A_n=\emptyset$ and $f([l,r])=I$ for every $l\in A_1$ and 
$r\in A_n$. 
	Also $f([l,r])=I$ for every $l\in A_i$ and $r\in A_j$ where $j-i\geq 2$.
\end{remark}

\begin{lemma}\label{lem:preimages}
	Let $J\subset I$ be a closed interval 
	and $f:I\to I$ a map with surjective intervals $A_1, \ldots A_n$, 
$n\geq 1$. 
	For every $i\in\{1, \ldots, n\}$ there exists an interval $J^i\subset 
A_i$ such that $f(J^i)=J$, 
	$f(\partial J^i)=\partial J$ and $J^i\cap R(A_i)\neq\emptyset$. 
\end{lemma}
\begin{proof}
Consider the interval $J=[a,b]$ and fix $i\in\{1, \ldots, n\}$. Let 
$a_i,b_i\in R(A_i)$ be such that $f(a_i)=a$ and  $f(b_i)=b$. Assume first that 
$b_i<a_i$ (see Figure~\ref{fig:JL}). Find the smallest $\tilde a_i>b_i$ such 
that $f(\tilde a_i)=a$. Then $J^i:=[b_i, \tilde a_i]$ has the desired 
properties. If $a_i<b_i$, then take $J^i=[a_i, \tilde b_i]$, where $\tilde 
b_i>a_i$ is the smallest such that $f(\tilde b_i)=b$. 
\end{proof}

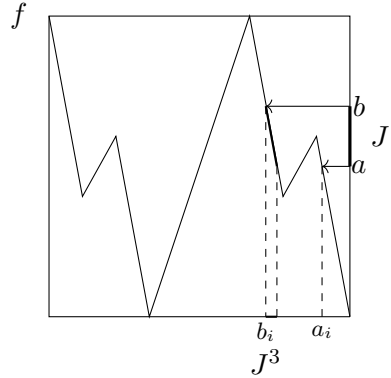
\begin{figure}[!ht]
	\centering
	\begin{tikzpicture}[scale=4]
	\begin{scope}
	[yscale=1,xscale=-1]
	\draw (0,0)--(1,0)--(1,1)--(0,1)--(0,0);
	\draw 
(0,0)--(1/9,0.6)--(2/9,0.4)--(1/3,1)--(2/3,0)--(7/9,0.6)--(8/9,0.4)--(1,1);
	\draw[very thick] (0,0.5)--(0,0.7);
	\node at (-0.1,0.6) {\small $J$};
	\draw[->] (0,0.5)--(0.5/5.4,0.5);
	\draw[->] (0,0.7)--(0.7/5.4+0.15,0.7);
	\draw[thick] (0.5/5.4+0.15,0.5)--(0.7/5.4+0.15,0.7);
	\draw[dashed] (0.5/5.4+0.15,0.5)--(0.5/5.4+0.15,0);
	\draw[dashed] (0.7/5.4+0.15,0.7)--(0.7/5.4+0.15,0);
	\draw[thick] (0.5/5.4+0.15,0)--(0.7/5.4+0.15,0);
	\node at (0.28,-0.15) {\small $J^3$};
	\draw[dashed] (0.5/5.4,0.5)--(0.5/5.4,0);
	\node at (0.5/5.4,-0.05) {\scriptsize $a_i$};
	\node at (0.7/5.4+0.15,-0.05) {\scriptsize $b_i$};
	\node at (-0.03,0.5) {\small $a$};
	\node at (-0.03,0.7) {\small $b$};
	\node at (1.1,1) {\small $f$};
	\end{scope}
	\end{tikzpicture}
	\caption{Construction of interval $J^i$ from the proof of 
Lemma~\ref{lem:preimages}.}
	\label{fig:JL}
\end{figure}

The following definition is a slight generalization of the notion of the 
``top'' of a permutation $p(G_f)$ of the graph $\Gamma_f$.

\begin{definition}
	Let $f: I\to I$ be a piecewise linear surjection and for a chain $C$ of 
$I$, let $p$ be a admissible 
	$C$-permutation of $G_f$.  For $x\in I$ denote
	the point in $p(G_f)$ corresponding to $f(x)$ by $p(f(x))$. 
	We say that $x$ is \emph{topmost in $p(G_f)$} if there exists a 
vertical ray $\{f(x)\}\times[h, \infty)$, where $h\in\R$, which intersects 
$p(G_f)$ only in $p(f(x))$.
\end{definition}

\begin{remark}
	If $A_1, \ldots, A_n$ are surjective intervals of $f: I\to I$, then 
every point in $R(A_n)$ is topmost. Also, for every $i=1, \ldots, n$ there 
exists a permutation of $G_f$ such that every point in $R(A_i)$ is topmost.
\end{remark}

\begin{lemma}\label{lem:twopoints}
	Let $f: I\to I$ be a map with surjective intervals $A_1, \ldots A_n$, 
$n\geq 1$. For $[a, b]=J\subset I$ and $i\in\{1, \ldots, n\}$ denote by $J^i$ 
an interval from Lemma~\ref{lem:preimages}. There exists an admissible 
permutation $p_i$ of $G_f$ such that both endpoints of $J^i$ are topmost in 
$p_i(G_f)$. 
\end{lemma}
\begin{proof}
Let $A_i=[l_i, r_i]$. Assume first that $f(l_i)=0$ and 
$f(r_i)=1$, thus $a_i<b_i$ (recall the notation 
	$a_i, \tilde a_i$ and $b_i, \tilde b_i$ from the proof of 
Lemma~\ref{lem:preimages}). Find the smallest critical point 
	$m$ of $f$ such that $m\geq \tilde b_i$ and note that $f(x)> f(a)$ for 
all $x\in A_i$, $x>m$. So we can reflect $f|_{[m,r_i]}$ over $f|_{[a_i,m]}$ and 
$f|_{[r_i,1]}$ over $f|_{[0,l_{i}]}$. 
	This makes $a_i$ and $\tilde b_i$ topmost, see 
Figure~\ref{fig:topmost}. In the case when $f(l_i)=1, f(r_i)=0$, thus $a_i>b_i$, 
we have that $f(x)<f(b)$ for all $x\in A_i$, $x>m$ so we can again reflect 
$f|_{[m,r_i]}$ over $f|_{[a_i,m]}$ making $\tilde a_i$ and $b_i$ topmost.
\end{proof}

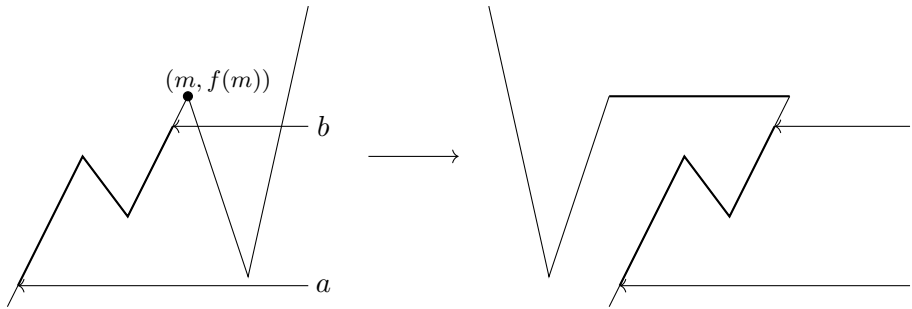
\begin{figure}[!ht]
	\centering
	\begin{tikzpicture}[scale=4]
	\draw (0,0)--(0.25,0.5)--(0.4,0.3)--(0.6,0.7)--(0.8,0.1)--(1,1);
	\draw[->] (1,0.07)--(0.035,0.07);
	\draw[->] (1,0.6)--(0.55,0.6);
	\draw[solid, fill] (0.6,0.7) circle (0.015);
	\node at (0.7,0.75) {\scriptsize $(m,f(m))$};
	\draw[thick] (0.55,0.6)--(0.4,0.3)--(0.25,0.5)--(0.035,0.07);
	\node at (1.05,0.6) {\small $b$};
	\node at (1.05,0.07) {\small $a$};
	
	\draw[->] (1.2,0.5)--(1.5,0.5);
	\draw (2,0)--(2.25,0.5)--(2.4,0.3)--(2.6,0.7);
	\draw[<-] (2.035,0.07)--(3,0.07);
	\draw[<-] (2.55,0.6)--(3,0.6);
	\draw[thick] (2.6,0.7)--(2,0.7);
	\draw (2,0.7)--(1.8,0.1)--(1.6,1);
	\draw[thick] (2.55,0.6)--(2.4,0.3)--(2.25,0.5)--(2.035,0.07);
	\end{tikzpicture}
	\caption{Making endpoints of $J^i$ topmost.}
	\label{fig:topmost}
\end{figure}

\begin{lemma}\label{lem:Mayerprep}
Let $X=\underleftarrow{\lim}\{I,f_i\}$, where each $f_i: I\to I$ is a 
continuous piecewise linear surjection and assume that $X$ is indecomposable. 
If $f_i$ contains at least three surjective intervals for every 
$i\in\N$, then there exist uncountably many planar embeddings of 
$X$ that are not strongly equivalent.
\end{lemma}

\begin{proof}
	For every $i\in\N$ let $k_i\geq 3$ be the number of surjective 
branches of $f_i$ and fix $L_i, R_i\in\{1, \ldots, k_i\}$ such that 
$|L_i-R_i|\geq 2$.
	Let $J\subset I$ and $(n_i)_{i\in\N}\in\prod_{i\in\N}\{L_i,R_i\}$. Then 
	$$J^{(n_i)}:=J\stackrel{\text{$f_1$}}{\leftarrow} 
J^{n_1}\stackrel{\text{$f_2$}}{\leftarrow} 
J^{n_1n_2}\stackrel{\text{$f_3$}}{\leftarrow} 
J^{n_1n_2n_3}\stackrel{\text{$f_4$}}{\leftarrow}\ldots$$ 
	is a well-defined subcontinuum of $X$. Here we used the notation 
$J^{nm}=(J^n)^m$. Moreover, Lemma~\ref{lem:twopoints} and 
Theorem~\ref{thm:algorithm} imply that $X$ can be embedded in the plane such 
that both points in $\partial J\leftarrow \partial J^{n_1}\leftarrow \partial 
J^{n_1n_2}\leftarrow \partial J^{n_1n_2n_3}\leftarrow\ldots$ are accessible. 
	
	Remark~\ref{rem:image} implies that for every $f: I\to I$ with 
surjective intervals $A_1, \ldots, A_n$, every $|i-j|\geq 2$ and every 
$J\subset I$ it holds that $f([J^i, J^j])=I$, where $[J^i, J^j]$ denotes the 
convex hull of $J^i$ and $J^j$. So if $(n_i), 
(m_i)\in\prod_{i\in\N}\{L_i,R_i\}$ 
	differ at infinitely many places,
	then there is no proper subcontinuum of $X$ which contains both $J^{(n_i)}$ 
and $J^{(m_i)}$, \ie they are contained in different composants of $X$. Now 
Theorem~\ref{thm:Maz} implies that there are uncountably many 
planar embeddings of $X$ that are not strongly equivalent.
\end{proof}

Next we prove that the assumption of at least three surjective intervals can be 
made without loss of generality for every nondegenerate indecomposable chainable continuum. For 
$X=\underleftarrow{\lim}\{I,f_i\}$, where each $f_i: I\to I$ is a continuous 
piecewise linear surjection, we show that there is 
$X'=\underleftarrow{\lim}\{I,g_i\}$ homeomorphic to $X$ such that $g_i$ has at 
least three surjective intervals for every $i\in\N$. We will build on the 
following remark.

\begin{remark}\label{rem:May} Assume that $f, g: I\to I$ each have at least two surjective 
intervals. Note that then $f\circ g$ has at least three surjective intervals. 
So if $f_i$ has two surjective intervals for every $i\in\N$, then $X$ can be 
embedded in the plane in uncountably many nonequivalent ways.
\end{remark}

\begin{definition}
	Let $\eps>0$ and let $f: I\to I$ be a continuous surjection. We say 
that $f$ is $P_{\eps}$ if for every two segments $A,B\subset I$ such that $A\cup 
B= I$ it holds that $d_{H}(f(A), I)<\eps$ or $d_{H}(f(B), I)<\eps$, where 
$d_{H}$ denotes the Hausdorff distance.
\end{definition}

\begin{remark}\label{rem:threepts}
	Let $f: I\to I$ and $\eps>0$. Note that $f$ is $P_{\eps}$ if and only 
if there exist $0\leq x_1<x_2<x_3\leq 1$ such that one of the following holds
	\begin{itemize}
		\item[(a)] $|f(x_1)-0|,|f(x_3)-0|<\eps$, $|f(x_2)-1|<\eps$, or
		\item[(b)] $|f(x_1)-1|,|f(x_3)-1|<\eps$, $|f(x_2)-0|<\eps$.
	\end{itemize}
\end{remark}

For $n<m$ denote by $f_n^m=f_n\circ f_{n+1}\circ\ldots\circ f_{m-1}$.

\begin{theorem}[Kuykendall \cite{Kuy}]\label{thm:kuykendall} The inverse limit 
$X=\underleftarrow{\lim}\{I,f_i\}$ is indecomposable if and only if for every $\eps>0$ and every 
$n\in\N$ there exists $m>n$ such that $f_n^m$ is $P_{\eps}$.
\end{theorem}

Furthermore, we will need the following strong theorem.

\begin{theorem}[Mioduszewski, \cite{Miod}]\label{thm:Miod} Two continua 
$\underleftarrow{\lim}\{I, f_i\}$ and $\underleftarrow{\lim}\{I, g_i\}$ are 
homeomorphic if and only if for every sequence of positive integers $\eps_i\to 
0$ there exists an infinite diagram as in Figure~\ref{fig:Miod},
	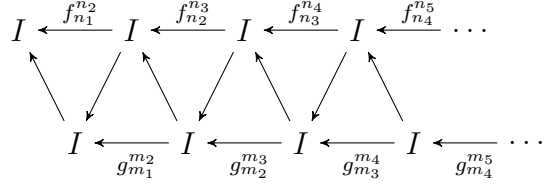
\begin{figure}[!ht]
		\centering
		\begin{tikzpicture}[->,>=stealth',auto, scale=1.5]
		\node (1) at (0,1) {$I$};
		\node (2) at (0.5,0) {$I$};
		\node (3) at (1,1) {$I$};
		\node (4) at (1.5,0) {$I$};
		\node (5) at (2,1) {$I$};
		\node (6) at (2.5,0) {$I$};
		\node (7) at (3,1) {$I$};
		\node (8) at (3.5,0) {$I$};
		\node (9) at (4,1) {$\ldots$};
		\node (10) at (4.5,0) {$\ldots$};
		\draw [->] (3) to (1);
		\draw [->] (5) to (3);
		\draw [->] (7) to (5);
		\draw [->] (9) to (7);
		\draw [->] (4) to (2);
		\draw [->] (6) to (4);
		\draw [->] (8) to (6);
		\draw [->] (10) to (8);
		\draw [->] (2) to (1);
		\draw [->] (3) to (2);
		\draw [->] (4) to (3);
		\draw [->] (5) to (4);
		\draw [->] (6) to (5);
		\draw [->] (7) to (6);
		\draw [->] (8) to (7);
		
		\node at (0.55,1.15) {\tiny $f_{n_1}^{n_2}$};
		\node at (1.55,1.15) {\tiny $f_{n_2}^{n_3}$};
		\node at (2.55,1.15) {\tiny $f_{n_3}^{n_4}$};
		\node at (3.55,1.15) {\tiny $f_{n_4}^{n_5}$};
		\node at (1.05,-0.2) {\tiny $g_{m_1}^{m_2}$};
		\node at (2.05,-0.2) {\tiny $g_{m_2}^{m_3}$};
		\node at (3.05,-0.2) {\tiny $g_{m_3}^{m_4}$};
		\node at (4.05,-0.2) {\tiny $g_{m_4}^{m_5}$};
		\end{tikzpicture}
		\caption{Infinite $(\eps_i)$-commutative diagram from 
Mioduszewski's theorem.}
		\label{fig:Miod}
	\end{figure} 
	where $(n_i)$ and $(m_i)$ are sequences of strictly increasing 
integers, $f_{n_i}^{n_{i+1}}=f_{n_i+1}\circ\ldots\circ f_{n_{i+1}}$, 
$g_{m_i}^{m_{i+1}}=g_{m_i+1}\circ\ldots\circ g_{m_{i+1}}$ for every $i\in\N$ 
and every subdiagram as in Figure~\ref{fig:Miod2} is $\eps_i$-commutative.
	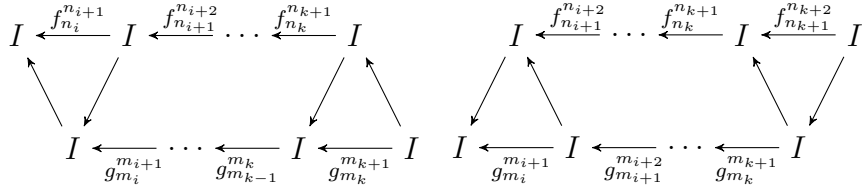
\begin{figure}[!ht]
		\centering
		\begin{tikzpicture}[->,>=stealth',auto, scale=1.5]
		\node (1) at (0,1) {$I$};
		\node (2) at (0.5,0) {$I$};
		\node (3) at (1,1) {$I$};
		\node (4) at (1.5,0) {$\ldots$};
		\node (5) at (2,1) {$\ldots$};
		\node (6) at (2.5,0) {$I$};
		\node (7) at (3,1) {$I$};
		\node (8) at (3.5,0) {$I$};
		\draw [->] (3) to (1);
		\draw [->] (5) to (3);
		\draw [->] (4) to (2);
		\draw [->] (6) to (4);
		\draw [->] (2) to (1);
		\draw [->] (3) to (2);
		\draw [->] (7) to (5);
		\draw [->] (7) to (6);
		\draw [->] (8) to (6);
		\draw [->] (8) to (7);
		
		\node at (0.55,1.15) {\tiny $f_{n_i}^{n_{i+1}}$};
		\node at (1.55,1.15) {\tiny $f_{n_{i+1}}^{n_{i+2}}$};
		\node at (2.55,1.15) {\tiny $f_{n_k}^{n_{k+1}}$};
		\node at (1.05,-0.2) {\tiny $g_{m_i}^{m_{i+1}}$};
		\node at (2.05,-0.2) {\tiny $g_{m_{k-1}}^{m_k}$};
		\node at (3.05,-0.2) {\tiny $g_{m_{k}}^{m_{k+1}}$};
		\end{tikzpicture}
		\begin{tikzpicture}[->,>=stealth',auto, scale=1.5]
		
		\node (2) at (0.5,0) {$I$};
		\node (3) at (1,1) {$I$};
		\node (4) at (1.5,0) {$I$};
		\node (5) at (2,1) {$\ldots$};
		\node (6) at (2.5,0) {$\ldots$};
		\node (7) at (3,1) {$I$};
		\node (8) at (3.5,0) {$I$};
		\node (9) at (4,1) {$I$};
		\draw [->] (5) to (3);
		\draw [->] (4) to (2);
		\draw [->] (6) to (4);
		\draw [->] (3) to (2);
		\draw [->] (7) to (5);
		\draw [->] (8) to (6);
		\draw [->] (8) to (7);
		\draw [->] (4) to (3);
		\draw [->] (9) to (8);
		\draw [->] (9) to (7);
		
		\node at (1.55,1.15) {\tiny $f_{n_{i+1}}^{n_{i+2}}$};
		\node at (2.55,1.15) {\tiny $f_{n_k}^{n_{k+1}}$};
		\node at (3.55,1.15) {\tiny $f_{n_{k+1}}^{n_{k+2}}$};
		\node at (1.05,-0.2) {\tiny $g_{m_i}^{m_{i+1}}$};
		\node at (2.05,-0.2) {\tiny $g_{m_{i+1}}^{m_{i+2}}$};
		\node at (3.05,-0.2) {\tiny $g_{m_{k}}^{m_{k+1}}$};
		\end{tikzpicture}
		\caption{Subdiagrams which are $\eps_i$-commutative for every 
$i\in\N$.}
		\label{fig:Miod2}
	\end{figure} 
\end{theorem}

\begin{theorem}\label{thm:Mayer}
Every nondegenerate indecomposable chainable continuum $X$ can be embedded in the 
plane in uncountably many ways that are not strongly equivalent.
\end{theorem}
\begin{proof}
	Let $X=\underleftarrow{\lim}\{I,f_i\}$, where each $f_i: I\to I$ is a 
continuous piecewise linear surjection. 
	If all but finitely many $f_i$ have at least three surjective 
intervals, we are done by Lemma~\ref{lem:Mayerprep}. 
	If for all but finitely many $i$ the map $f_i$ has two surjective 
intervals, we are done by Remark~\ref{rem:May}.
	
	Now fix a sequence $(\eps_i)$ such that $\eps_i>0$ for every $i\in\N$ 
and $\eps_i\to 0$ as $i\to\infty$. 
	Fix $n_1=1$ and find $n_2>n_1$ such that $f_{n_1}^{n_2}$ is 
$P_{\eps_1}$. Such $n_2$ exists by 
	Theorem~\ref{thm:kuykendall}. For every $i\in\N$ find $n_{i+1}>n_i$ 
such that $f_{n_i}^{n_{i+1}}$ is $P_{\eps_i}$. 
	The continuum $X$ is homeomorphic to $\underleftarrow{\lim}\{I, 
f_{n_i}^{n_{i+1}}\}$. Every $f_{n_i}^{n_{i+1}}$ is 
	piecewise linear and there exist $x_1^i<x_2^i<x_3^i$ as in 
Remark~\ref{rem:threepts}. Take them to be critical 
	points and assume without loss of generality that they satisfy condition 
$(a)$ of Remark~\ref{rem:threepts}. 
	Define a piecewise linear surjection $g_i: I\to I$ with the same set of 
critical points as $f_{n_i}^{n_{i+1}}$ 
	such that $g_i(c)=f_{n_i}^{n_{i+1}}(c)$ for all critical points 
$c\not\in\{x_1, x_2, x_3\}$ and $g_i(x_1)=g_i(x_3)=0$, $g_i(x_2)=1$. Then $g_i$ 
is $\eps_i$-close to $f_{n_i}^{n_{i+1}}$. By Theorem~\ref{thm:Miod}, 
$\underleftarrow{\lim}\{I, f_{n_i}^{n_{i+1}}\}$ is homeomorphic to 
$\underleftarrow{\lim}\{I, g_i\}$. Since every $g_i$ has at least two 
surjective intervals, this finishes the proof by Remark~\ref{rem:May}.
\end{proof}

\begin{remark}
	Specifically, Theorem~\ref{thm:Mayer} proves that the pseudo-arc has 
uncountably many embeddings that are not strongly equivalent. Lewis 
\cite{Lew} has already proven this with 
	respect to the standard version of equivalence, by carefully constructing 
embeddings with different prime end structures. 
\end{remark}

In the next theorem we expand the techniques from this section to construct 
uncountably many strongly nonequivalent 
embeddings of every chainable continuum that contains a nondegenerate indecomposable subcontinuum. 
First we give a generalisation of Lemma~\ref{lem:twopoints}.

\begin{lemma}\label{lem:indecsubc}
	Let $f: I\to I$ be a surjective map and let $K\subset I$ be a closed 
interval. Let $A_1, \ldots, A_n$ be the surjective intervals of $f|_K: K\to 
f(K)$, and let $J^i$, $i \in \{1,\dots, n\}$, be intervals from 
Lemma~\ref{lem:preimages} applied to the map $f|_K$. \\
	Assume $n\geq 4$. Then there exist $\alpha, \beta\in\{1, \ldots, n\}$ 	
	such that $|\alpha-\beta|\geq 2$ and such that there exist admissible 
permutations $p_{\alpha}, p_{\beta}$ of $G_f$ such that both endpoints of 
$J^{\alpha}$ are topmost in $p_{\alpha}(G_{f|_K})$ and such that both endpoints 
of $J^{\beta}$ are topmost in $p_{\beta}(G_{f|_K})$. 
\end{lemma}

\begin{proof}
Let $K=[k_l, k_r]$ and $f(K)=[K_l, K_r]$. Let $x>k_r$ be the 
smallest local extremum of $f$ such that $f(x)>K_r$ or $f(x)<K_l$. A surjective 
interval $A_i=[l_i, r_i]$ will be called increasing (decreasing) if 
$f(l_i)=K_l$ ($f(r_i)=K_l$).
	
	{\bf Case 1.} Assume $f(x)>K_r$ (see Figure~\ref{fig:extperm}). If 
$A_i=[l_i, r_i]$ is increasing, since $f(x)>K_r$, there exists an admissible 
permutation which reflects $f|_{[m, x]}$ over $f|_{[a_i, m]}$ and leaves 
$f|_{[x, 1]}$ fixed. Here $m$ is chosen as in the proof of 
Lemma~\ref{lem:twopoints}. Since there are at least four surjective intervals, 
	at least two are increasing. This finishes the proof.
	
	{\bf Case 2.} If $f(x)<K_l$ we proceed as in the first case but for 
decreasing $A_i$.
\end{proof}

\begin{figure}[!ht]
	\centering
	\begin{tikzpicture}[scale=4]
	\draw (0,0)--(0,1)--(1,1)--(1,0)--(0,0);
	\draw[dashed] (0,0)--(0,-0.2);
	\draw[dashed] (1,0)--(1,-0.2);
	\draw[dashed] (0,0)--(-0.2,0);
	\draw[dashed] (0,1)--(-0.2,1);
	\node at (0,-0.3) {\small $k_l$};
	\node at (1,-0.3) {\small $k_r$};
	\node at (-0.3,0) {\small $K_l$};
	\node at (-0.3,1) {\small $K_r$};
	\node at (0.4,0.5) {$\ldots$};
	\draw 
(0.5,0.5)--(0.6,0)--(0.7,0.6)--(0.8,0.4)--(0.9,1)--(1,0.7)--(1.1,1.1)--(1.2,
-0.1);
	\draw[thick] (0.633,0.2)--(0.69,0.55);
	\node at (0.7,0.65) {\scriptsize $f(m)$};
	\node at (1.1,1.15) {\scriptsize $f(x)$};
	\draw[->] (1.3,0.5)--(1.45,0.5);
	
	\draw 
(1.6,0.5)--(1.6,0)--(2.1,0)--(2.1,0.6)--(2,0.6)--(2,0.4)--(1.9,0.4)--(1.9,
1)--(1.8,1)--(1.8,0.7)--(1.7,0.7)--(1.7,1.1)--(2.3,1.1)--(2.3,-0.1)--(2.4,-0.1);
	\node at (2.05,0.65) {\scriptsize $f(m)$};
	\node at (2,1.15) {\scriptsize $f(x)$};
	\draw[thick] (2.1,0.2)--(2.1,0.55);
	\end{tikzpicture}
	\caption{Permuting in the proof of Lemma~\ref{lem:indecsubc}.}
	\label{fig:extperm}
\end{figure}
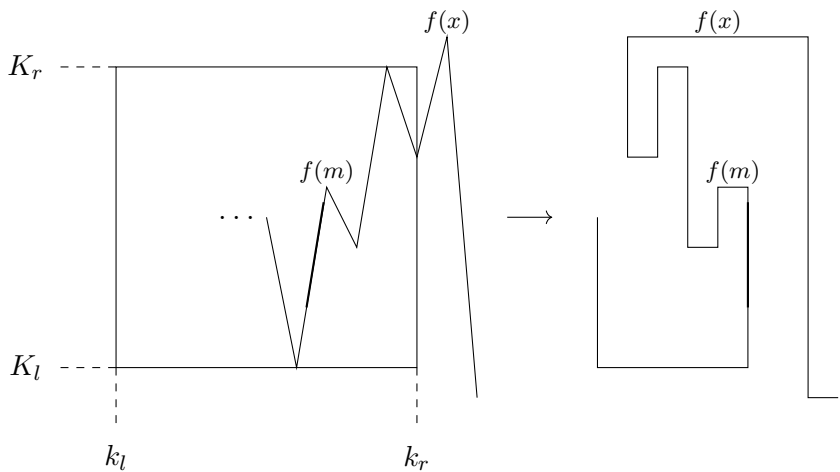

\begin{theorem}\label{thm:Mayersubc}
	Let $X$ be a chainable continuum that contains a nondegenerate indecomposable 
subcontinuum $Y$. Then $X$ can be embedded in the plane in uncountably many 
ways that are not strongly equivalent.
\end{theorem}

\begin{proof}
Let 
$$
Y:=Y_0\stackrel{\text{$f_1$}}{\leftarrow} 
Y_1\stackrel{\text{$f_2$}}{\leftarrow} Y_2\stackrel{\text{$f_3$}}{\leftarrow} 
Y_3\stackrel{\text{$f_4$}}{\leftarrow}\ldots.
$$ 
If $\phi, \psi: X\to \R^2$ are strongly equivalent planar embeddings of $X$, 
then $\phi|_Y, \psi|_Y$ are strongly equivalent planar embeddings of $Y$. We will 
construct uncountably many strongly nonequivalent planar embeddings of $Y$ extending 
to planar embeddings of $X$, which will complete the proof. 
	
	According to Theorem~\ref{thm:kuykendall} and Theorem~\ref{thm:Miod} we 
can assume that $f_i|_{Y_i}: Y_i\to Y_{i-1}$ has at least four surjective 
intervals for every $i\in\N$. For a closed interval $J\subset Y_{j-1}$, let
 $\alpha_j, \beta_j$ be integers from Lemma~\ref{lem:indecsubc} applied to $f_j: Y_j\to 
Y_{j-1}$, and denote the appropriate subintervals of $Y_j$ by $J^{\alpha_j}$, 
$J^{\beta_j}$. For every sequence $(n_i)_{i\in\N}\in\prod_{i\in\N}\{\alpha_i, 
\beta_i\}$ we obtain a subcontinuum of $Y$:
	$$J^{(n_i)}:=J\stackrel{\text{$f_1$}}{\leftarrow} 
J^{n_1}\stackrel{\text{$f_2$}}{\leftarrow} 
J^{n_1n_2}\stackrel{\text{$f_3$}}{\leftarrow} 
J^{n_1n_2n_3}\stackrel{\text{$f_4$}}{\leftarrow}\ldots$$ 
	We use the notation of the proof of Lemma~\ref{lem:Mayerprep}. 
Lemma~\ref{lem:indecsubc} implies that for every sequence $(n_i)$ there exists an 
embedding of $Y$ such that both points of $\partial J\leftarrow \partial 
J^{n_1}\leftarrow \partial J^{n_1n_2}\leftarrow \partial 
J^{n_1n_2n_3}\leftarrow\ldots$ are accessible and which can be extended to an 
embedding of $X$. This completes the proof.
\end{proof}

We have proven that every chainable continuum containing a nondegenerate indecomposable 
subcontinuum has uncountably many embeddings that are not strongly equivalent. Thus we pose the 
following question.

\begin{question}\label{q:above}
 Which hereditarily decomposable chainable continua have uncountably many 
planar embeddings that are not equivalent and/or strongly 
equivalent? 
\end{question}

\begin{remark}
	Mayer has constructed uncountably many nonequivalent 
planar embeddings (in both senses) in \cite{May}  of the $\sin\frac 1x$ continuum by varying 
the rate of convergence of the ray. 
	This approach readily generalizes to any Elsa continuum.
	We do not know whether the approach can be generalized to all chainable 
continua which contain a dense ray. Specifically, it would be interesting to 
	see if $\underleftarrow{\lim}\{I, f_{Feig}\}$ (where $f_{Feig}$ denotes 
the logistic interval map at the Feigenbaum parameter) 
can be embedded in uncountably many nonequivalent ways. 
	However, this approach would not generalize to the remaining 
hereditarily decomposable chainable continua since there 
	exist hereditarily decomposable chainable continua which do not contain a dense 
ray, see \eg \cite{Jan}. 
\end{remark}

\begin{remark}\label{rem:n_emb}
In Figure~\ref{fig:n_embed} we give examples of planar continua which 
have exactly $n\in\N$ or countably many nonequivalent planar embeddings. 
However, except for the arc, all the examples we know are not chainable. 
\end{remark}

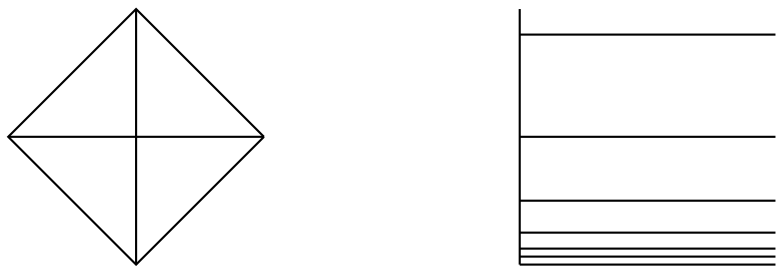
\begin{figure}[!ht]
	\centering
	\begin{tikzpicture}[scale=1.7]
	%\draw[thick] (1,1) circle (1);
	\draw[thick] (1,1)--(2,1);
	\draw[thick] (1,1)--(0,1);
	\draw[thick] (1,1)--(1,2);
	\draw[thick] (1,1)--(1,0);
	\draw[thick] (2,1)--(1,0)--(0,1)--(1,2)--(2,1);
	\thicklines
	\draw[thick] (4, 0)--(4,2);
	\draw[thick] (4, 0)--(6,0);
	\draw[thick] (4, 1.8)--(6,1.8);
	\draw[thick] (4, 1)--(6,1);
	\draw[thick] (4, 0.5)--(6,0.5);
	\draw[thick] (4, 0.25)--(6,0.25);
	\draw[thick] (4, 0.125)--(6,0.125);
	\draw[thick] (4, 0.0625)--(6,0.0625);
	\end{tikzpicture}
	\caption{Left: Planar projection (Schlegel diagram) of the sides of the 
pyramid with $n\geq 4$ faces has exactly $n$ embeddings that are not strongly equivalent, determined by the choice of the unbounded face. Actually any 
planar representation of a polyhedron with $n$ faces would do in the previous 
example. We are indebted to Imre P\'eter T\'oth  for these examples. Continua 
with exactly $n=2,3$ nonequivalent planar embeddings in the strong sense are 
\eg the letters $H, X$ respectively. In the standard sense, there is only one planar 
embedding of each of these examples.
\newline
Right: the harmonic comb has countably many nonequivalent 
embeddings in both senses; any finite number of 
non-limit 
		teeth can be flipped over to the left to produce a 
nonequivalent embedding.} 
	\label{fig:n_embed}
\end{figure}

\begin{question}
	Is there a non-arc chainable continuum for which there exist at most countably 
	many nonequivalent planar embeddings? 	
\end{question}

\begin{remark}\label{rem:otherdef}	
For inverse limit spaces $X$ with a single {\bf unimodal} bonding map 
that are not hereditarily decomposable, 
Theorems~\ref{thm:Mayer} and \ref{thm:Mayersubc}
hold with the standard notion of equivalence as well, for details see \cite{embed}.
This is because every self-homeomorphism of $X$ is known to be 
pseudo-isotopic (two self-homeomorphisms $f, g$ of $X$ are called {\em pseudo-isotopic} if 
$f(C)=g(C)$ for every composant $C$ of $X$) to a power of the shift 
homeomorphism 
(see \cite{BBS}), and so every composant can only be mapped to one in a countable 
collection of composants. 
Hence, if uncountably many composants can be made accessible in at least two points, 
then there are uncountably many nonequivalent embeddings. In general there are no such rigidity results on the 
group of self-homeomorphisms of chainable continua. For example, there are 
uncountably many self-homeomorphisms of the pseudo-arc up to pseudo-isotopy, 
since it is homogeneous and all arc-components are degenerate. Thus we ask the following question.
\end{remark}

\begin{question}
For which indecomposable chainable continua is the group of all 
self-homeomorphisms up to pseudo-isotopy at most countable?
 \end{question}

\end{document}